\documentclass[reqno,11pt]{amsart}
\usepackage{latexsym}
\usepackage{amsmath,amsfonts,amssymb,epsfig}
\usepackage{amsthm, amscd}
\usepackage{hyperref}
\usepackage{xcolor}
\usepackage{setspace}
\usepackage{overpic}
\usepackage{eucal}
\usepackage{mathrsfs}

\newtheorem{thm}{Theorem}
\newtheorem*{mthm}{Main Theorem}
\newtheorem{cor}[thm]{Corollary}
\newtheorem{lem}[thm]{Lemma}

\newtheorem{prop}[thm]{Proposition}

\theoremstyle{definition}
\newtheorem{rem}[thm]{Remark}

\theoremstyle{definition}

\newtheorem{exm}[thm]{Example}

\newcommand{\no}{\noindent}

\def\R{\mathbb{R}}

\def\sign{\text{\rm sign}}

\numberwithin{equation}{section} 
\numberwithin{thm}{section} 

\newcommand{\vvcenteredinclude}[2]{\begingroup
\setbox0=\hbox{\includegraphics[scale=#1]{#2}}%
\parbox{\wd0}{\box0}\endgroup}  

\title[Tree invariants and Milnor linking numbers]{Tree invariants and Milnor linking numbers\\ with indeterminacy}

\author{R. Komendarczyk}
\address{
Tulane University,
New Orleans, Louisiana 70118 } 
\email{rako@tulane.edu}

\thanks{Supported by NSF DMS 1043009 and DARPA YFA N66001-11-1-4132 during the years 2011-2015} 

\author{A. Michaelides}

\address{
University of South Alabama,
Mobile, AL 36688 } 
\email{amichaelides@southalabama.edu}

\begin{document}

\begin{abstract}
	 The paper concerns the {\em tree invariants} of string links, introduced by  Kravchenko and Polyak, which are closely related to the classical Milnor linking numbers also known as $\bar{\mu}$--invariants. We prove that, analogously as for $\bar{\mu}$--invariants, certain residue classes of tree invariants yield link homotopy invariants of closed links. The proof is arrow diagramatic and provides a more geometric insight into the indeterminacy through certain tree stacking operations. Further, we show that the indeterminacy of tree invariants is consistent with the original Milnor's indeterminacy. For practical purposes, we also provide a recursive procedure for computing arrow polynomials of tree invariants.
 \end{abstract}

\maketitle
\vspace{-.5cm}
\section{Introduction}\label{S:intro}
Arrow polynomial formulas of Polyak and Viro \cite{Polyak-Viro:1994} are a computationally attractive way to represent Vassiliev's finite type invariants of knots and links \cite{Vassiliev:1992, Birman-Lin:1993, Bar-Natan:1995a}.  The input to such formula is a {\em Gauss diagram} $G_L$ of a (based) link $L=L_1\cup\ldots\cup L_n$, $L_i:S^1\longmapsto \R^3$ in $\R^3$, obtained from any plane projection of $L$, by drawing $n$ disjoint oriented circles with  basepoints, and marked (positive/negative) arrows connecting distinct points on the circles. Points of the $i$th circle correspond to values of a parameter for  $L_i:S^1\longmapsto \R^3$, c.f. \cite{Goussarov-Polyak-Viro:2000}. A positive/negative arrow between two points on the $i$th and $j$th circle is drawn, if and only if, for the corresponding parameter values the plane diagram of $L$ has a positive/negative crossing, the arrow points from the underpass to the overpass. Alternatively, we may replace the components of $G_L$ i.e. the  circles, with vertical or horizontal oriented segments, we call {\em strings}, assuming that the beginning and end of each string is identified with the basepoint, Figure \ref{fig:gauss-d} illustrates this situation.
\begin{figure}[ht]
	\includegraphics[height=.22\textheight]{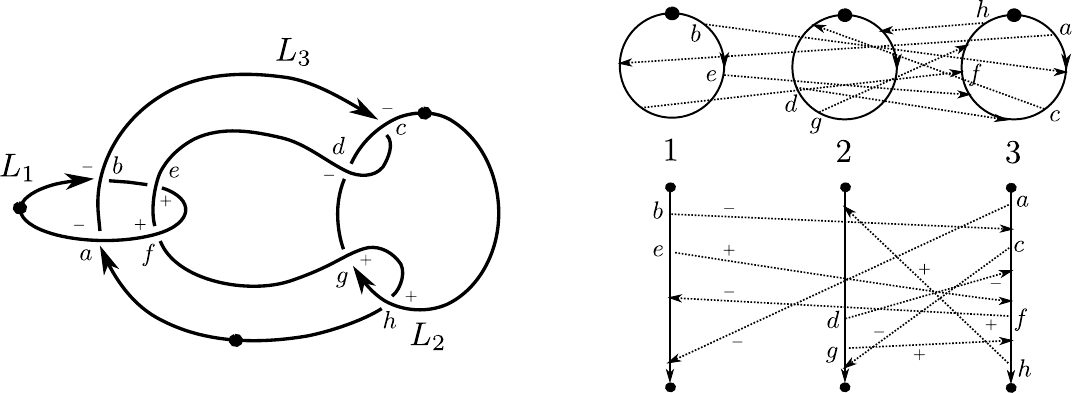}
	\caption{Gauss diagrams of Borromean rings (drawn with {\em circle}  and respectively {\em string} components).}\label{fig:gauss-d}
\end{figure}

 An {\em arrow diagram} $A$ is an arbitrary unmarked diagram (i.e. arrows have no signs attached to them) analogous to a Gauss diagram but not necessarily obtained from a plane link diagram.  
The arrow diagrams and Gauss diagrams can be paired as follows; for  a  Gauss diagram $G=G_L$ of the $n$--component link $L$, a {\em{embedding}} of an  arrow diagram $A$ in $G$ is a graph embedding of $A$ into $G$ mapping components of $A$ to components of $G$, preserving the basepoints and arrow orientations. Define the {\em{sign}} of an embedding $\phi : A \longrightarrow G$  by\footnote{arrows of $A$ will be denoted by the greek letters: $\alpha$, $\beta$,\ldots , and arrows of $G$ lowercase letters: $g$, $h$,\ldots}
\begin{equation}\label{eq:sign_representation}
\sign(\phi) = \prod_{\alpha \in A} \sign(\phi (\alpha)),
\end{equation}
\no where the $\sign(\phi(\alpha))$ is a sign of the arrow $g=\phi(\alpha)$ in $G$. Then $\langle A, G\rangle$ stands for the sum 
\begin{equation}\label{eq:<A,G>}
\langle A, G\rangle = \sum_{\phi: A \rightarrow G} \sign(\phi),
\end{equation}
\no taken over all embeddings $\phi: A \longrightarrow G$ of $A$ in $G$. An arrowhead of $\alpha\in A$ will be denoted by $h(\alpha)$ and the arrowtail by $t(\alpha)$. We write $\alpha\sim (i,j)$ or $g\sim (i,j)$ if $\alpha$, resp. $g$, has its head on the $i$--component of $A$, and tail on the $j$--component. 
A formal sum of arrow diagrams $P=\sum_i c_i A_i$ with integer coefficients is known as an {\em arrow polynomial} \cite{Polyak-Viro:1994}, and $\langle P,G\rangle$ is defined from \eqref{eq:sign_representation} and \eqref{eq:<A,G>} by the linear extension. A theorem of Goussarov \cite{Goussarov-Polyak-Viro:2000} shows that any finite type invariant $v$ of knots can be expressed  as $\langle P_v,\,\cdot\,\rangle$ for a suitable choice of the arrow polynomial $P_v$. Arrow polynomials of some low degree invariants have been computed in \cite{Ostlund:2004,Polyak-Viro:1994,Polyak-Viro:2001,Willerton:2002}. For instance, the second coefficient of the Conway polynomial $c_2(K)$ of a knot $K$, represented by a Gauss diagram $G_K$ is given by $\langle\vvcenteredinclude{.2}{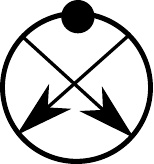}\,, G_K\rangle$ (c.f. \cite{Polyak-Viro:1994}). Apart from low degree examples, the arrow diagram formulae are known for: the coefficients of the Conway \cite{Chmutov-Khoury-Rossi:2009} and the HOMFLY-PT polynomials, \cite{Chmutov-Polyak:2010}.

In the case of {\em string links} \cite{Habegger-Lin:1990,Bar-Natan:1995b}, Kravchenko and Polyak \cite{Kravchenko-Polyak:2011} introduced a family of link homotopy invariants, called {\em tree invariants} which are closely related to  the classical Milnor linking numbers, \cite{Milnor:1954,Milnor:1957,Levine:1988}.
\begin{figure}[!ht] 
	\centering
	$\vcenter{\hbox{\includegraphics[width=0.15\textwidth]{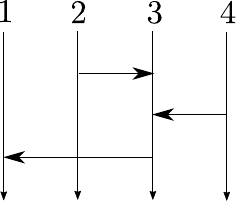}}}$
	\qquad\quad 
	$\vcenter{\hbox{\includegraphics[width=0.15\textwidth]{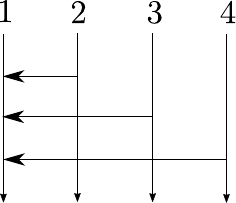}}}$
	\qquad\quad
	 $\vcenter{\hbox{\includegraphics[width=.35\textwidth]{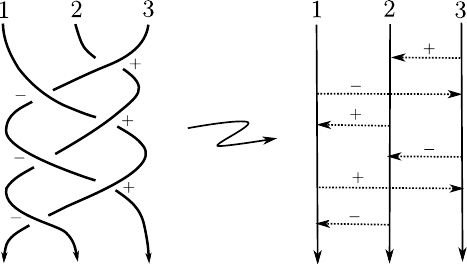}}}$
	\caption{Sample tree diagrams in $\mathcal{A}(I;1)$, $I=\{1,2,3,4\}$ (left). Gauss diagrams of a Borromean string link (right).}\label{fig:trees}
\end{figure} 
 Recall that homotopy or link--homotopy is a weaker form of equivalence than isotopy which, in particular, allows self-crossing of strands.
 Specifically, Kravchenko and Polyak defined a family  of  {\em (planar) tree diagrams} on $n$ components with a distinguished component $j$ called a {\em trunk}, for $I = \{ i_1, i_2,...,i_r \}$,  $1 \leq i_1 < i_2 < ...< i_r \leq n$ and $j \in I$, denote the set of these diagrams with leaves  indexed by $I$ and a trunk on $j$ by\footnote{this is a slightly adjusted notation from \cite{Kravchenko-Polyak:2011} (Figure \ref{fig:trees}(left)), where the index of the trunk is also in $I$.} $\mathcal{A}(I;j)$, referring to Section \ref{S:tree-invariants} for further details.
 Paraphrasing the main theorem of \cite{Kravchenko-Polyak:2011} we have
\begin{thm}[Kravchenko and Polyak \cite{Kravchenko-Polyak:2011}]\label{thm:Z_I;j-invariants}
	Let $\ell=(\ell_1,\ldots,\ell_n)$ be an $n$--component string link and $G_\ell$ its Gauss diagram. Consider the following arrow polynomial 
	\begin{equation}\label{eq:Z_I;j-poly}
	Z_{I;j}=\sum_{A\in \mathcal{A}(I;j)} \text{\rm sign}(A) A,
	\end{equation}
then 	
	\begin{equation}\label{eq:Z_I;j}
	 Z_{I;j}(\ell)=\langle Z_{I;j}, G_\ell\rangle=\sum_{A\in \mathcal{A}(I;j)} \text{\rm sign}(A)\langle A,G_\ell\rangle,
	\end{equation}
	defines a link-homotopy invariant of the string link $\ell$, where $\text{sign}(A)=(-1)^q$ and $q$ is given by  number of arrows in $A$ pointing to the right.
\end{thm}

\no  Invariants $Z_{I;j}(\ell)$ are called the {\em tree invariants} and are finite type string link invariants, \cite{Chmutov-Duzhin-Mostovoy:2012}. As such, they must necessarily be determined  by the classical $\mu$--invariants of string links,  \cite{Habegger-Masbaum:2000}. In their paper  \cite{Kravchenko-Polyak:2011} Kravchenko and Polyak show that for a particular order of leaves and the trunk, namely for $I=\{i_1,i_2,\ldots,i_r\}$ and $j$ such that $1\leq j <i_1 <i_2 <\ldots <i_r \leq n$, we have the identity
\begin{equation}\label{eq:Z_I;j=mu_I;j}
 Z_{I;j}(\ell)=\mu_{I;j}(\ell).
\end{equation}

The main theorem of the current paper is a generalization of Theorem \ref{thm:Z_I;j-invariants} to closed $n$--component links. An analogous result by Polyak and Viro,  Theorem 6 in \cite{Polyak-Viro:1994} (see also \cite{Ostlund:2004}), concerns the case $n=3$. 
\begin{mthm}
	Suppose $L$ is a based $n$--component link and $G_L$ its Gauss diagram, 
	for $I=\{i_1,\ldots,i_r\}$, $1\leq i_1<\ldots <i_r\leq n$ and $j=i_k$
  the following quantity defines a homotopy invariant of $L$;
	\begin{equation}\label{eq:bar-Z_I;j}
	\overline{Z}_{I;j}(L)=\langle Z_{I;j}, G_L\rangle \mod \Delta_Z(I;j),
	\end{equation}  
	where
	\begin{equation}\label{eq:Delta_Z(I;j)}
	 \Delta_Z(I;j)=\gcd\bigl\{ \langle Z_{J;k}, G_L\rangle\ |\  J\subsetneq I; k\in J\bigr\}.
	\end{equation} 
\end{mthm}
Further, the relation of $\overline{Z}_{I;j}$--invariants and Milnor higher linking numbers i.e. $\overline{\mu}$--invariants, \cite{Milnor:1954, Milnor:1957} is obtained in the following.
\begin{cor}\label{cor:bar-mu-bar-Z}
	For $I=\{i_1,i_2,\ldots,i_r\}$ and $j=i_1$ such that $1\leq <i_1 <i_2 <\ldots <i_r \leq n$ we have
	\begin{equation}
	 \overline{Z}_{I;j}(L)=\overline{\mu}_{I-\{j\};j}(L).
	\end{equation}
\end{cor}

Recall that the $\overline{\mu}$--invariants are important invariants of homotopy links. In particular, they are capable of distinguishing $3$--component homotopy links as proven in \cite{Milnor:1954}. The indeterminacy of $\overline{\mu}$--invariants was  studied by several  authors most notably by Levine, see \cite{Levine:1988a, Levine:1988} and references therein. It is not known however, if an appropriate form of indeterminacy, would provide a complete set of numerical link homotopy invariants. More recently Kotorii \cite{Kotorii:2013} defines $\overline{\mu}$--invariants for Turaev's nanophrases \cite{Turaev:2007}, via Magnus expansion adapted to nanophrases. In contrast, our definition is  a direct generalization of Kravchenko--Polyak tree invariants  and geared towards further applications in geometric knot theory, \cite{Komendarczyk-Michaelides:2016}. For a related connection between trees  and Milnor invariants of string links one may refer to the recent work in \cite{Koytcheff-Volic:2015}, and other relevant approaches, e.g. in \cite{Mellor-Melvin:2003} and \cite{Turaev:1976}.

The techniques presented in this paper are geometric and depend on certain natural tree diagram decompositions, related ideas can be found in the work of {\" O}stlund \cite{Ostlund:2004}. As a practical outcome we also  obtain recursive procedure for tree invariants (previously computed directly for $n=3$ and $n=4$ in \cite{Kravchenko-Polyak:2011}) which yields a convenient computational  algorithm (see Proposition \ref{prop:recursive}).

The paper is organized as follows; In Section \ref{S:milnor}, we review the original construction of $\overline{\mu}$--invariants for links and string links as presented in  \cite{Milnor:1957,Levine:1988}. The construction of tree invariants by Kravchenko and Polyak \cite{Kravchenko-Polyak:2011} is reviewed in Section \ref{S:tree-invariants}. Section \ref{S:recursive} provides definitions of natural tree stacking operations and related technical results needed in Section \ref{S:proof}, it also includes the above mentioned algorithm for generating arrow polynomials of tree invariants for any $n$. Proof of Main Theorem and Corollary \ref{cor:bar-mu-bar-Z} is given in Section \ref{S:proof}.

 We finally mention, that the applications of Milnor linking numbers are fairly broad; including distant areas such as topological fluid dynamics or plasma physics, c.f. \cite{Evans-Berger:1992,Laurence-Stredulinsky:2000, DeTurck-Gluck-Komendarczyk-Melvin:2013a,Komendarczyk:2009}. In the forthcoming paper: \cite{Komendarczyk-Michaelides:2016}, we show how the arrow diagrammatic formulation of linking numbers can be applied to address a  geometric question of Freedman and Krushkal \cite{Freedman-Krushkal:2014}, concerning estimates for thickness of $n$--component links.

\section{Acknowledgements}
Both authors acknowledge the generous support of NSF DMS 1043009 and DARPA YFA N66001-11-1-4132 during the years 2011-2015. Both authors wish to thank the anonymous referee for the detailed report which, among other things, pointed out interpretations in the language of operads, see Remark \ref{rem:operad}. 

A version  of Main Theorem, without a precise characterization of indeterminacy, was obtained by the second author in his doctoral thesis, \cite{Michaelides:2015}. 

\section{Linking numbers of closed links and string links}\label{S:milnor}

In this section we review the construction of 
Milnor invariants,  \cite{Milnor:1954, Milnor:1957}. Denote  by  $L=L_1\cup\ldots\cup L_n$ an $n$--component ordered, oriented link in $S^3$. Recall that any diagram of $L$ yields the Wirtinger presentation of $\pi=\pi_1(S^3-L)$, where the generators are {\em meridians}, one for each arc in the diagram and the relations are derived from the crossings in the diagram, \cite{Hatcher:2002}. Different meridians of a given component are conjugate to each other in $\pi$, and a choice of basepoint on the $j$--component $L_j$ indicates a preferred meridian $m_j$. In $\pi$ we also distinguish the {\em parallels}, i.e. push-offs of $L_j$ denoted by $l_j$ satisfying $\text{lk}(l_j,L_j)=0$. Consider $F=F(m_1,\ldots, m_n)$, the free group generated by the preferred meridians  $\{m_j\}$. In \cite{Milnor:1954} Milnor proved that the universal homomorphism $F\longmapsto \pi$ descends to an epimorphism of the lower central series quotients: 
\[
\phi:F/F_q\longmapsto \pi/\pi_q,\qquad  \text{for any $q$},
\]
 (in fact, due to the later result of Stallings \cite{Stallings:1965}, $\phi$ is an isomorphism). Recall that given a group $G$, the lower central series is given as $G_1=G$, $G_2=[G,G_1]$, \ldots $G_p=[G,G_{p-1}]$,\ldots . 
 Thus, for any $q$ there exists $l^q_j\in F(m_1,\ldots, m_n)$ representing the parallel $l_j$ modulo the $q$th stage  of the lower central series of $\pi$  (i.e. $\phi(l^q_j F_q)=l_j \pi_q$). Every element of $F$ can be regarded as a unit in\footnote{the ring of power series in $n$ non-commuting variables $X_i$.}  $\mathbb{Z}\langle X_1,...X_n\rangle$  via the {\em Magnus expansion}, which is a ring homomorphism
\begin{equation}\label{eq:theta_F}
\begin{split}
\theta_F: &\mathbb{Z}F \longrightarrow \mathbb{Z}\langle X_1,...X_n\rangle,\\
  & m_i \longmapsto 1+X_i,\qquad m_i^{-1} \longmapsto 1 - X_i + X_i^2 - X_i^3+ \ldots
\end{split}
\end{equation} 
 embedding $F$ into $\mathbb{Z}\langle X_1,...X_n\rangle$ as its group of units. Given the $j$th parallel $l_j$, as above, we have the expansion
\begin{equation}\label{eq:l^q_j-expansion}
 \theta_F(l^q_j)=1+\sum_{\{i_1,\ldots, i_s\}\subset [n]; s\geq 1} \mu_{i_1\,\ldots\, i_s;j}\, X_{i_1} X_{i_2}\ldots X_{i_s},\qquad [n]=\{1,\ldots,n\}.
\end{equation}
The coefficients $\mu_{\mathtt{I};j}=\mu_{i_1,\ldots, i_s;j}$ are defined for each ordered sequence of integers $\mathtt{I}=(i_1,\ldots, i_s)$ $1\leq i_k\leq n$. Following \cite{Milnor:1957}, let
\begin{equation}\label{eq:Gamma-mu}
\begin{split}
\Gamma_\mu(\mathtt{I};j)  =\{ & \mu(k_1,k_2,...,k_{r-1};k_r) \ |\ \text{where $\{k_1,k_2,...,k_{r-1},k_r\}$, $2 \leq r < s$ ranges over }\\ &\quad \text{all subsequences of $(i_1, i_2, ...,i_s,j)$ obtained by deleting at least one }\\
&\quad \text{of its elements and permuting the remaining elements cyclically}\},
\end{split}
\end{equation}
and
\begin{equation}\label{eq:bar-mu-Delta(I)}
 \overline{\mu}_{\mathtt{I};j} \equiv \mu_{\mathtt{I};j} \mod\ \Delta_\mu(\mathtt{I};j),\qquad\quad \Delta_\mu(\mathtt{I};j)=\gcd(\Gamma_\mu(\mathtt{I};j)).
\end{equation}
 In \cite{Milnor:1957}, Milnor proved that, for $s<q$, $\mathtt{I}=\{i_1,i_2,\ldots,i_s\}$, the residue classes $\overline{\mu}_{\mathtt{I};j}$ are isotopy\footnote{in fact they are concordance invariants as follows from the result of Stallings in \cite{Stallings:1965}} invariants of $L$, and if the indices in $\{\mathtt{I},j\}$ are all distinct, $\overline{\mu}_{\mathtt{I};j}$ are link-homotopy invariants. The residue classes $\overline{\mu}_{\mathtt{I};j}(L)=\overline{\mu}_{\mathtt{I};j}$ are commonly known as the {\em Milnor linking numbers} or {\em $\overline{\mu}$--invariants}, and $\Delta_\mu(\mathtt{I};j)$ is called the {\em indeterminacy}. One obvious property of 
$\overline{\mu}$--invariants is equivariance under permutations $\gamma\in\Sigma_n$,  i.e.
\begin{equation}\label{eq:mu-perm-indices}
\overline{\mu}_{i_1\,i_2\,\ldots\, i_r; j}(L^\gamma)=\overline{\mu}_{\gamma(i_1)\,\gamma(i_2)\,\ldots\, \gamma(i_r); \gamma(j)}(L), 
\end{equation}
where $L^\gamma$ is the link $L$ with permuted components; $L^\gamma_i=L_{\gamma(i)}$.
Further relations are proven in \cite[p. 294]{Milnor:1957}, for example cyclic symmetry: 
\begin{equation}\label{eq:bar-mu-cyclic}
\overline{\mu}_{i_1,i_2,\ldots,i_{r-1};i_r}(L)=\overline{\mu}_{i_2,i_3,\ldots,i_{r};i_1}(L).
\end{equation}
\no In fact,  due to cyclic symmetry \eqref{eq:bar-mu-cyclic}, we may consider a smaller set 
\begin{equation}\label{eq:Gamma'-mu}
	\begin{split}
		\Gamma'_\mu(\mathtt{I};j) & =\{ \mu(k_1,k_2,...,k_{r-1};k_r) \ |\ \text{where $\{k_1,k_2,...,k_{r-1},k_r\}$, $2 \leq r < s$ ranges}\\ &\quad\quad \text{over all proper subsequences of $(i_1, i_2, ...,i_s,j)$}\},\\
		 &\quad\quad \text{and}\ \Delta'_\mu(\mathtt{I};j):=\gcd(\Gamma'_\mu(\mathtt{I};j)).
	\end{split}
\end{equation}
From the basic properties of $\text{gcd}$, one obtains 
\[
 \Delta'_\mu(\mathtt{I};j)=\Delta_\mu(\mathtt{I};j).
\]

For {\em string links}, the construction of $\mu$--invariants is completely analogous. Recall that an $n$--component string link (see Figure \ref{fig:trees}(right) for an example of a string link diagram) is a smooth embedding $\ell:\sqcup^{n}_{k=1} I_k\longmapsto D^2\times I$, of $n$ copies $I_1$,\ldots $I_n$ of the unit interval $I$ into the cylinder $C=D^2\times I$, such that each $\sigma_k=\sigma|_{I_k}$ is anchored at the bottom and top of the cylinder at fixed points $\{a_k\}$, i.e. for each $k=1,\ldots,n$;
\[
 \ell_k(0)=(a_k,0),\qquad \ell_k(1)=(a_k,1).
\]
Each string link $\ell$ can be closed up into a link $L=\widehat{\ell}$ by adding unlinked connecting strands outside the cylinder $D^2\times I\subset \R^3$, this {\em closure} operation is denoted by $\widehat{\,\cdot\,}$ in \cite{Habegger-Lin:1990}. As before, one considers the group $\pi=\pi_1(C\setminus \sigma)$, this time there is a canonical choice of {\em meridians} $\{m_k\}$ represented by loops in $D^2\times\{1\}$ based at a fixed point $x_0\in \partial(D^2\times\{1\})$, with $\text{lk}(m_k,\ell_k)=+1$. One also defines {\em canonical parallels}; $l_k$ as loops in $C$ based at $x_0$ and closed up by fixed arcs in the boundary of $C$, with $\text{lk}(l_k,\ell_k)=0$. Again each parallel $l_j$ has its expansion \eqref{eq:l^q_j-expansion}, modulo the $q$--stage of the lower central series of $F$ and $\pi$. Differently from the case of closed links, the coefficients 
$\mu_{\mathtt{I};j}(\ell)=\mu_{I;j}$ do not require the indeterminacy and yield the isotopy invariants of string links, \cite{Levine:1988} and for distinct indices in $\{\mathtt{I};j\}$, they define link-homotopy invariants. Given a link $L=\widehat{\ell}$, obtained as the closure of a string link $\ell$, we have the following identity, \cite{Levine:1988}:
\begin{equation}\label{eq:bar-mu-sigma}
\overline{\mu}_{\mathtt{I};j}(L) \equiv \mu_{\mathtt{I};j}(\ell) \mod\ \Delta_\mu(\mathtt{I};j).
\end{equation}

We end this section by pointing out that the linking numbers $\mu_{\mathtt{I}; j}$ significantly depend on the order of
integers in  $(\mathtt{I};j)=(i_1\, i_2\,\ldots \, i_r,j)$ (e.g., in general $\mu_{i_1\,i_2\,\ldots \,i_r; j} (\ell) \neq \mu_{i_2\, i_1\,\ldots\, i_r ; j} (\ell)$.)

\section{Tree invariants}\label{S:tree-invariants}
In Section \ref{S:intro}, we introduced planar tree diagrams and tree invariants, in this section we provide their formal definitions, following the source in \cite{Kravchenko-Polyak:2011}. 
 
Let\footnote{We assume a slightly different convention than in \cite{Kravchenko-Polyak:2011}, see Remark \ref{rem:index-notation}.} $I = \{ i_1, i_2,...,i_r \}$,  $1 \leq i_1 < i_2 < ...< i_r \leq n$ and $j\in  I$, a {\em{tree diagram}} $A$ with {\em leaves} on the components numbered by $I$ and a trunk on the $j$--component, is an arrow diagram which satisfies the following conditions:

\begin{enumerate}
	\item[({\bf{d1}})]  An arrowtail $t(\alpha)$ and an arrowhead $h(\alpha)$ of an arrow $\alpha\in A$ belong to different strings;
	\item[({\bf{d2}})]   There is exactly one arrow with an arrowtail on the $i$-th string, if $i \in I-\{j\}$ and no such arrow if $i\notin I-\{j\}$;
	\item[({\bf{d3}})]   All arrows have arrowheads on strings indexed by $I$;
	\item[({\bf{d4}})]   All arrowheads precede the (unique) arrowtail for each $i \in I-\{j\}$, as we follow the $i$-th string along its (downward) orientation.
\end{enumerate}

Define the {\em{degree}} of an arrow diagram $A$ to be the total number of arrows in the diagram, and note that it is always equal to $r-1$.
One may visualize any tree diagram $A$, as the rooted tree graph: $T(A)$ obtained from $A$ by removing the part of each of the  components that lies below the corresponding (unique) arrowtail (by {\bf (d2)} and {\bf (d4)}). Then, $T(A)$ is graph--isomorphic to a  rooted tree with $r$ leaves and the root on the $j$--component. 
Every tree  is a planar graph, however we will refer to a tree diagram $A$ as {\em{planar}}, if in its planar realization the order of leaves coincides with the initial ordering $i_1 < i_2 <... <i_l < j < i_{l+1}< ... <i_{r}$ of the components. Note that the above axioms imply that every arrow $\alpha\in A$ is uniquely determined by its coordinates $(i,j)$. In Figure \ref{fig:trees-planar}, there are two planar and a non-planar tree diagrams, together with the trees obtained from each of them. In the remaining part of this paper, for practical purposes, we will make no distinction between $T(A)$ and $A$.
\begin{figure}[!ht] 
	\centering
	$\vcenter{\hbox{\includegraphics[width=0.22\textwidth]{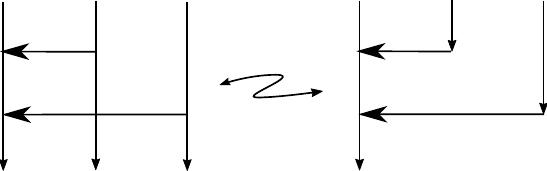}}}$\quad\qquad
	$\vcenter{\hbox{\includegraphics[width=0.6\textwidth]{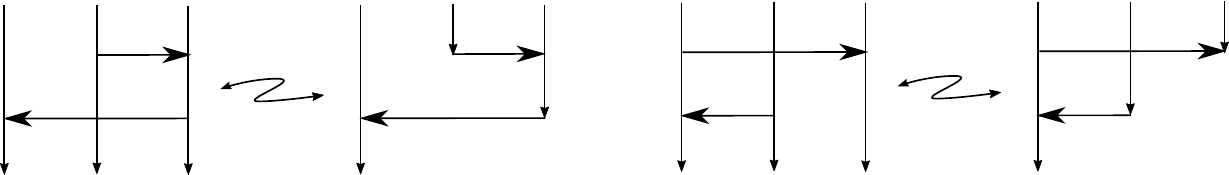}}}$
	\caption{From a diagram $A$ to a tree $T(A)$ (left). Planar and non-planar tree diagrams (middle and right).}\label{fig:trees-planar}
\end{figure} 
As in Section \ref{S:intro}, $\mathcal{A}(I; j)$ stands for the set of all planar tree diagrams with leaves on $I$ and the trunk on the $j$-th component and 
\[
\mathcal{A}_j = \cup_{I} \mathcal{A}(I;j).
\]
\begin{rem}[Notation]\label{rem:index-notation}
	Differently than in \cite{Kravchenko-Polyak:2011}, we will treat the trunk of a tree diagram as one of its leaves.  It makes only a small notational difference and simplifies further considerations.
	For instance $Z_{1,2,3,4;2}=Z_{134;2}$, where $Z_{134;2}$ agrees with Kravchenko and Polyak convention (which, whenever used, will skip the comma separators). Further, we also abbreviate
	\[
	\mathcal{A}(n;j):=\mathcal{A}([n];j),\qquad Z_{n;j}:=Z_{[n];j}=Z_{12\ldots\widehat{j}\ldots n;j},
	\] 
	where $[n]=\{1,2,\ldots, n\}$.
\end{rem}
The definition of tree invariants of \cite{Kravchenko-Polyak:2011} involves the notion of disassociative algebras \cite{Loday-Vallette:2012,Loday-Frabetti-Chapoton-Goichot:2001}, recall that diassociative algebra over a ground field $k$ is a $k$-space $V$ equipped with two $k$-linear maps
\[
 \vdash\,:V\otimes V\longrightarrow V,\qquad \dashv\,:V\otimes V\longrightarrow V,
\]
which satisfy the following properties
\begin{equation}\label{eq:Dias-rel}
\begin{split}
(x\dashv y)\dashv z=x\dashv (y\dashv z), & \quad (x\dashv y)\dashv z=x\dashv (y\vdash z),\\
(x\vdash y)\dashv z=x\vdash (y\dashv z), & \quad (x\dashv y)\vdash z=x\vdash (y\vdash z), \quad (x\vdash y)\vdash z=x\vdash (y\vdash z).
\end{split}
\end{equation}
Depicting products $x\vdash y$, $x\dashv y$ as elementary trees, compositions
of operations $\vdash$, $\dashv$ can be interpreted as tree {\em grafting}, \cite{Loday-Vallette:2012}. Following, \cite{Kravchenko-Polyak:2011}, consider $Dias(n)$ to be the quotient of a vector space generated by planar rooted
trees with $n$ leaves: $\mathcal{A}(n)=\bigcup_j \mathcal{A}(n;j)$ by the relations \eqref{eq:Dias-rel} of the algebra.

Let $G_\ell$ be a Gauss diagram of a string link $\ell$, see Figure \ref{fig:trees}(right) for an example.  Kravchenko and Polyak  \cite{Kravchenko-Polyak:2011}, consider the following element of a quotient algebra $Dias(n)$, \cite[p. 306]{Kravchenko-Polyak:2011},
\begin{equation}
Z_j(G_\ell)=\sum_{A \in \mathcal{A}_j} \sign(A)\langle A, G_\ell \rangle \cdot [A], 
\end{equation}
\no where $[A]$ denotes an equivalence class in $Dias(n)$, the pairing $\langle A, G_\ell \rangle$ is defined in \eqref{eq:<A,G>}, and 
\begin{equation}\label{eq:sign(A)-def}
\text{\rm sign}(A) =(-1)^q,\qquad q=\#\ \text{of arrows in $A$ pointing to the right},
\end{equation}
is called the {\em sign of the diagram $A$}. 

\begin{thm}[Kravchenko and Polyak, \cite{Kravchenko-Polyak:2011}]\label{thm:Z_j}
Let $\ell$ be an $n$--component string link and $G_\ell$ its Gauss diagram. Then $Z_j (\ell)=Z_j (G_\ell)$ is a $Dias(n)$--valued homotopy invariant of $\ell$.  
\end{thm}

Since the equivalence class $[A]$ of a tree $A$, with a trunk on the $j$-th component, depends only on the set of its leaves, all $A \in \mathcal{A}(I, j)$ represent the same equivalence class and the sum given by   
\[
Z_{I, j}(G_\ell) = \sum_{A \in \mathcal{A}(I;j)} \sign(A) \langle A, G_\ell \rangle
\]
 is a homotopy invariant as stated by Theorem \ref{thm:Z_I;j-invariants} of Section \ref{S:intro}. 
Invariants $\{Z_{I, j}(\ell)\}$ are finite type thus by \cite[Corollary 6.4]{Habegger-Masbaum:2000} they can be expressed as polynomials in Milnor $\mu$-invariants of string links defined in Section \ref{S:milnor}. 
 Because the index $\mathtt{I}=(i_1,\ldots,i_{n-1})$ for $\{\mu_{\mathtt{I};j}\}$ is ordered and 
$I=\{i_1,\ldots,i_{r}\}$ in $Z_{I;j}$ has the increasing order, by assumption, in general: $Z_{I;j}(\ell)\neq \mu_{I-\{j\};j}(\ell)$. However, if $I=\mathtt{I}$, i.e. $\mathtt{I}$ has the same order as $I$, then we have the following result
\begin{thm}[Kravchenko and Polyak, \cite{Kravchenko-Polyak:2011}]\label{thm:tree=mu}
	Let $\ell$ be an ordered string link on $n$ strings and let $1 \leq j = i_1 < i_2 < ... < i_r \leq n$. Then
	\begin{equation}\label{eq2:Z_I;j=mu_I;j}
	Z_{I; j} (\ell) = \mu_{i_2 ... i_r;j} (\ell).
	\end{equation} 
\end{thm}
\no (see also Corollary \ref{cor:Z-mu-prod}). 
The proof of Theorem \ref{thm:tree=mu} follows from the fact that the tree invariants $Z_{I, j}$ and $\mu_{I,j}$ satisfy the same skein relation (see \cite{Polyak:2005}), and have the same normalization, i.e  $Z_{I;j} (\ell) = \mu_{I-\{j\}; j} (\ell) = 0$ for all string links $\ell$ with the $j$-th component passing in front of all the others.  Further relations for tree invariants were proven in \cite[Proposition 4.2]{Kravchenko-Polyak:2011} and, for convenience, are stated below (where $I=\{i_1,\ldots,i_r\}$, $1\leq i_1<\ldots<i_r\leq n$).
\begin{itemize}
	\item[({\bf s1})] For any $1<k<r$ and  $I^+=\{i_1,\ldots,i_{k-1},i_k\}$, $I^-=\{i_k,i_{k+1},\ldots,i_r\}$,
	\begin{equation}\label{eq:s1}
	Z_{I;i_k}(\ell)=Z_{I^+;i_k}(\ell)\,Z_{I^-;i_k}(\ell).
	\end{equation}
	\item[({\bf s2})] Let $\overline{\ell}$ be a string link obtained from $\ell$ by  reflecting the ordering, i.e. $\overline{\ell}_i=\ell_{\overline{i}}$, where $\overline{i}=n+1-i$. Then,
	\begin{equation}\label{eq:s2}
	Z_{I;j}(\overline{\ell})=(-1)^r Z_{\overline{I};\overline{j}}(\ell).
	\end{equation}
	\item[({\bf s3})] 	Given a cyclic permutation $\sigma=(i_1\, i_2\,\ldots \, i_r)$, let $\ell^\sigma$ be	a string link with renumbered strings according to $\ell^\sigma_i=\ell_{\sigma(i)}$. Then
	\begin{equation}\label{eq:s3}
	Z_{I;i_r}(\ell^\sigma)=Z_{I;i_1}(\ell).
	\end{equation}
\end{itemize}
\begin{rem}\label{rem:Z_I;j-product}
Note that ({\bf s1}) shows that a computation of $Z_{I;j}(\ell)$ with an arbitrary trunk $j\in I$, can be reduced to two cases: {\em 1.} the trunk is on the first component and {\em 2.} the trunk is on the last component.
\end{rem}
At the end of this section let us exhibit some lower degree invariants $Z_{I;j}$ (for $n=2 \: \text{and} \: 3$), as in \cite[p. 308]{Kravchenko-Polyak:2011}, and express them in terms of $\mu$--invariants. 

In the case where $n = 2$, we have the following tree invariants (using our notational conventions):
\begin{equation}
Z_{1,2;1} (\ell)=Z_{2;1} (\ell) = \Bigl\langle\vvcenteredinclude{0.27}{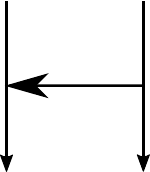}, G_\ell\Bigr\rangle, \quad Z_{1,2;2} (\ell)=Z_{1;2} (\ell) = - \Bigl\langle\vvcenteredinclude{0.27}{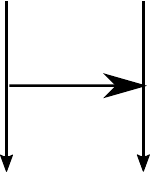},  G_\ell\Bigr\rangle,
\end{equation}
which agree with the linking number: 
\begin{equation}\label{eq:Z_2;1=lk}
Z_{2;1} (\ell) = \text{lk}(\ell_1 , \ell_2) = - Z_{1;2} (\ell).
\end{equation}
 For diagrams with two arrows, we have
\medskip
\begin{equation}\label{eq:tree-invariants-n=3}
\begin{split}
Z_{2 3 ;1} (\ell) = \Bigl\langle&\vvcenteredinclude{0.3}{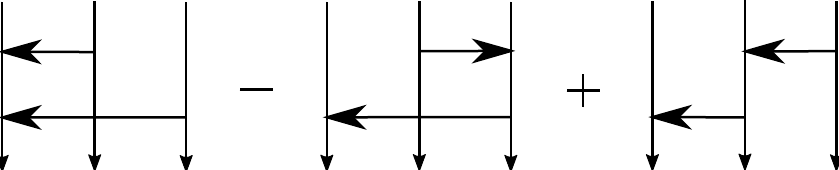}, G_\ell\Bigr\rangle,  \quad Z_{13 ;2} (\ell) = -  \Bigl\langle\vvcenteredinclude{0.3}{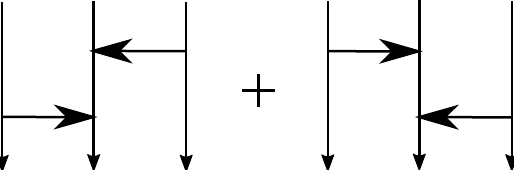}, G_\ell\Bigr\rangle,\\
&Z_{ 12 ;3} (\ell) = \Bigl\langle\vvcenteredinclude{0.3}{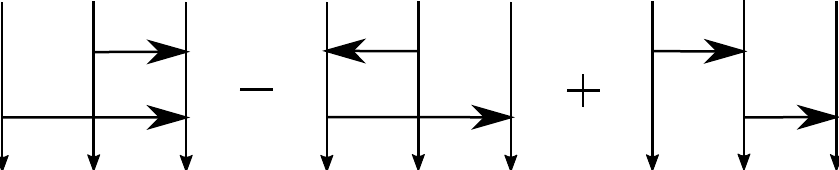}, G_\ell\Bigr\rangle .
\end{split}
\end{equation}
\smallskip

\no As shown in \cite[Proposition 4.2]{Kravchenko-Polyak:2011}, we have the following identities for the above invariants
\begin{equation}\label{eq:Z_123-mu_123}
 \begin{split}
 Z_{2 3 ;1} (\ell) & =\mu_{2 3; 1}(\ell),\\
 Z_{1 3 ;2} (\ell) & =Z_{1;2} (\ell) Z_{3;2} (\ell)= -\mu_{1; 2}(\ell)\mu_{3;2 }(\ell),\\
 Z_{1 2 ;3} (\ell) & =Z_{1 2 ;3} ((\ell^{\sigma^{-1}})^\sigma)=Z_{2 3;1}(\ell^{\sigma^{-1}})=\mu_{2 3; 1}(\ell^{\sigma^{-1}})=\mu_{\sigma(2) \sigma(3); \sigma(1)}(\ell)=\mu_{3 1; 2}(\ell)
 \end{split}
\end{equation}
The first equation is just \eqref{eq2:Z_I;j=mu_I;j}. In the second identity, we applied ({\bf s2}) and \eqref{eq:Z_2;1=lk}, in the third identity ({\bf s3}) and \eqref{eq:mu-perm-indices} with $\sigma=(1\, 2\, 3)$. The invariants $\mu_{3 1; 2}(\ell)$ and $\mu_{2 3; 1}(\ell)$ are not equal and in general differ by a sign and a sum of products of pairwise linking numbers of $\ell$. The computation in \eqref{eq:Z_123-mu_123} can be easily generalized as follows.
\begin{cor}\label{cor:Z-mu-prod}
	Let $I=\{i_1,\ldots,i_r\}$, $1\leq i_1<\ldots<i_r\leq n$, if $1<k<r$ then
	\begin{equation}\label{eq:Z-mu-prod}
		Z_{I;i_k}(\ell)=\mu_{i_3\,i_4\,\ldots\,i_k\,i_1;i_2}(\ell)\,\mu_{i_{k+1}\,\ldots\, i_r;i_k}(\ell).
	\end{equation}
	In particular for $k=1$ we obtain $Z_{i_2\,\ldots\,i_r;i_1}(\ell)=\mu_{i_2\,\ldots\,i_r;i_1}(\ell)$ and for $k=r$:   $Z_{i_1\,\ldots\,i_{r-1};i_r}(\ell)=$ $\mu_{i_3\,i_4\,\ldots\,i_k\,i_1;i_2}(\ell)$.
\end{cor}
\no We end this section with a computational example.
\begin{exm}
Let $\ell$ be the string link shown in Figure \ref{fig:trees}(right) along with its Gauss diagram $G_\ell$, and let us compute $Z_{1,2,3;1}(\ell)=Z_{23;1}(\ell)$ using identities in  \eqref{eq:tree-invariants-n=3}. Note that $G_\ell$ contains three subdiagrams of the first type, two of which contribute $+1$ and one of which contributes $-1$. Also note that $G_\ell$ does not contain any  subdiagrams of the other two types, as a result $Z_{23;1}(\ell) = 1$. 
\end{exm}

\section{Recursive construction of tree invariants}\label{S:recursive}

  We begin by introducing a certain useful operation, we refer to as {\em tree stacking}, apart from the fact, that tree stacking is our primary tool in proving  Main Theorem, it is essential to obtain a recursive description of  planar trees, which is presented in the second part of this section.
  
Recall, that $\mathcal{A}(n;r)$ is the set of planar tree diagrams satisfying axioms ({\bf{d1}})--({\bf{d4}}) of Section \ref{S:tree-invariants}. Tree stacking operations are indexed by the leaves and defined as
\begin{equation}\label{eq:stack_k}
\begin{split}
\prec_k:\mathcal{A}(n;r) & \times \mathcal{A}(m;s)\longrightarrow \mathcal{A}(m+n-1;t),\\
(P,Q) & \longrightarrow P\prec_k Q,
\end{split}
\end{equation}
where $k\in [n]$ indexes a fixed leaf of a tree in $\mathcal{A}(n;r)$ and 
$P\prec_k Q$ is simply obtained by ``gluing'' the trunk of $Q$ on the $k$th leaf of $P$, as shown in Figure \ref{fig:tree-stacking}. 
The index $t$ of the trunk of $P\prec_k Q$ is determined as follows:
\begin{equation}\label{eq:P-Q-trunk}
t=\begin{cases}
r; & \text{if } r< k,\\
r+m-1; & \text{if } r> k,\\
r+s-1; & \text{if } r=k.
\end{cases}
\end{equation}
\begin{figure}[ht]	
	\includegraphics[width=.4\textwidth]{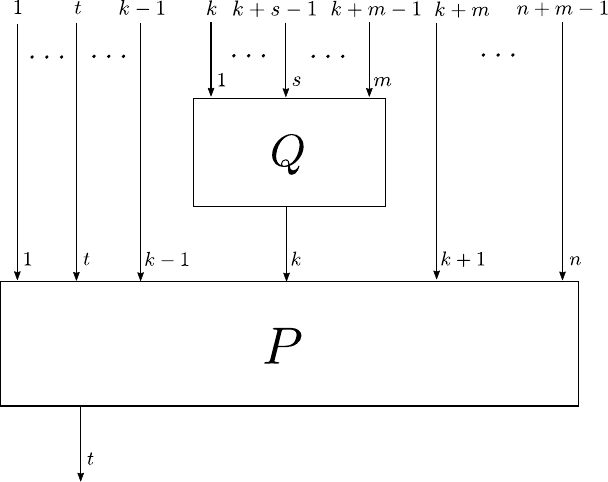}
	\caption{Stacking $Q\in \mathcal{A}(m;s)$ onto the $k^\text{th}$ leaf of $P\in \mathcal{A}(n;r)$ yields $P\prec_k Q\in \mathcal{A}(m+n-1;t)$.}\label{fig:tree-stacking}
\end{figure}
\no Operations $\prec_k$ are clearly well defined, i.e. produce planar trees on $n+m-1$ leaves with the trunk on  $t$. We also set 
\[
 P\prec_k \varnothing=P,\qquad \varnothing\prec_k P=P. 
\] 
Given $P\in \mathcal{A}(I;r)$, define
\begin{equation}\label{eq:r-l-trunk}
\begin{split}
 r(P) & =\text{number of leaves of $P$ to the {\em right} of the trunk},\\
 l(P) & =\text{number of leaves of $P$ to the {\em left} of the trunk}.
 \end{split}
\end{equation}
And the {\em total number of leaves}:
\[
  |P|=r(P)+l(P)+1.
\]
It will be useful, to keep track of the leaves indexing in $P$ and $Q$ after stacking. For $Z=P\prec_k Q$, $n=|P|$, and $m=|Q|$, we define the following multindices (as in Figure \ref{fig:tree-stacking})
\begin{equation}\label{eq:I(P;Z)-I(Q;Z)}
\begin{split}
I(P;Z) &=[1,\ldots, k-1]\,\cup\{k+s-1\}\cup\, [k+m,\ldots,m+n],\\ 
I(Q;Z) & =[k,\ldots,k+m-1],
\end{split}
\end{equation}
where $I(X;Z)$ indexes the leaves of $X$ in $Z$. Suppose a tree diagram $A\in \mathcal{A}(n;t)$ admits a decomposition (where the parenthesis are inserted in an arbitrary way)
\begin{equation}\label{eq:A=B_1-prec-B_k}
 A=B_1\prec_{b_1} B_2\prec_{b_2} \ldots \prec_{b_{k-1}} B_{k},\qquad B_i\in \mathcal{A}(m_i;t_i).
\end{equation}
Let $j\in [n]$ the index of  $j$th leaf component of $A$ ($j$--component in short), we say that the $j$th leaf {\em stems} from $B_i$'th factor in $A$, if and only if, it was added at the moment of stacking the $B_i$'th factor in  \eqref{eq:A=B_1-prec-B_k}. Denote
\begin{equation}\label{eq:i(j,B_i;A)}
\begin{split}
	i(j,A; B_i) & =\text{the index of $j$th leaf of $A$ in the $[m_j]$ index interval of $B_j$},\\
	i(k, B_i; A) & =\text{the index of $k$th leaf of $B_j$ in the index interval $[n]$ of $A$}.
\end{split}
\end{equation}
Note that  $i(j,A;B_i)$ is only  well defined for $j\in I(B_i;A)$, i.e. $j$th leaf stems from $B_i$. From \eqref{eq:sign(A)-def} we also obtain the following sign identity,  
\begin{equation}\label{eq:sign-P-Q}
 \sign(P\prec_k Q)=\sign(P)\,\sign(Q).
\end{equation}

 Building blocks of trees in $\mathcal{A}(n;r)$ are the unique elements of $\mathcal{A}(2;1)$ and $\mathcal{A}(2;2)$  shown in Figure \ref{fig:elem-trees} and called {\em elementary trees}.
\begin{figure}[ht]
\includegraphics[width=.1\textwidth]{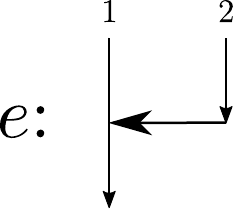}\qquad\qquad \includegraphics[width=.1\textwidth]{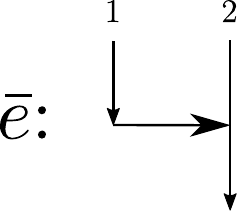}
\caption{Elementary trees $e\in \mathcal{A}(2;1)$ and $\bar{e}\in \mathcal{A}(2;2)$.}\label{fig:elem-trees}
\end{figure}
Whenever it is clear from the context the horizontal arrow of $e$ or $\bar{e}$ will also be denoted by $e$ or $\bar{e}$.  The following lemma conveys a basic fact that trees can be recursively constructed by adding leaves.
\begin{lem}\label{lem:A(n;r)-recursive}
Suppose $A\in \mathcal{A}(n;r)$, then there is $k\in [n-1]$ such that 
\begin{equation}\label{eq:A-A'-prec-e}
 A=A'\prec_k e,\quad\text{for}\quad A'\in \begin{cases}
 \mathcal{A}(n-1;r), & k\geq r;\\
 \mathcal{A}(n-1;r-1), & k<r;
 \end{cases}
\end{equation}
 or
\begin{equation}\label{eq:A-A'-prec-bar-e}
A=A'\prec_k \bar{e},\quad\text{for}\quad  A'\in \begin{cases}
\mathcal{A}(n-1;r), & k> r;\\
\mathcal{A}(n-1;r-1), & k\leq r.
\end{cases}
\end{equation}
\end{lem}
\begin{proof}
	We will induct with respect to $n\geq 2$. For $n=2$, the claim is obvious, since $\mathcal{A}(1;1)=\varnothing$, $(n-1)<2$ and the only tree in $\mathcal{A}(2;1)$ is given by $\varnothing\prec_1 e=e$ and in $\mathcal{A}(1;2)$ the only tree is $\varnothing\prec_1 \bar{e}=\bar{e}$. For the inductive step, note that for any tree diagram  $A\in \mathcal{A}(n+1;r)$, in its planar realization as shown in Figure \ref{fig:trees-planar}, horizontal arrows can be pushed up or down, so that there is at most one arrow at each vertical level of the diagram. Ordering the arrows, from the top to bottom, we let $\alpha_{top}\sim (i,j)$ be a top arrow. The top arrow has the obvious property that there are no arrowheads above $h(\alpha_{top})$ along the $i$--component, and equally are no  arrowheads above $t(\alpha_{top})$ along the $j$--component of $A$ (by ({\bf d2}) there can be no arrowtail above an arrowhead). Note that $\alpha_{top}$ generally depends on the choices of heights, i.e. the vertical ordering of horizontal arrows along components. Because of planarity, we claim that $\alpha_{top}$ is ``short'' i.e. 
	\begin{equation}\label{eq:j}
	 j=\begin{cases}
	i+1, &\quad \text{if}\ i<j;\\
	i-1, &\quad \text{if}\ i>j.
	\end{cases}
	\end{equation}
	Indeed, if $i<j$ and $j>i+1$, then there is  a $k$--component, such that $i<k<j$. The vertical edge of $k$--component, would have to intersect the edge $\alpha_{top}$, because the first arrowhead along this component is below $h(\alpha_{top})$. This contradicts planarity of $A$ and proves the first case of \eqref{eq:j}, the second case of \eqref{eq:j} can be shown analogously. 
	
	Now, thanks to the property ({\bf d4}), there is no arrowhead/tail along the $j$--component, below the tail of $\alpha_{top}$. Therefore, removing the vertical edge corresponding to that component together with $\alpha_{top}$, yields a tree $A'$ which is a subgraph of $A$. From the definition of $\prec_k$ in \eqref{eq:stack_k}  we obtain \eqref{eq:A-A'-prec-e} or \eqref{eq:A-A'-prec-bar-e}
	proving the inductive step. \qedhere
\end{proof}
\no Clearly, the expansions: $A'\prec_k e$, and $A''\prec_j \bar{e}$, for  some $A'$ and $A''$ may produce isomorphic diagrams. Lemma \ref{lem:A(n;r)-recursive} implies the following 

\begin{prop}\label{prop:recursive}
 Sets $\mathcal{A}(n;r)$ are recursively defined as follows
\[
\begin{split}
\mathcal{A}(1;1) =\varnothing,\ 
\mathcal{A}(n;r) & =\{ A\prec_k e, A\prec_j \bar{e}\ |\ A\in \mathcal{A}(n-1;r-1), k=1,\ldots, r-2; j=1,\ldots, r-1\}\\
& \cup \{ A\prec_k e, A\prec_j \bar{e}\ |\ A\in \mathcal{A}(n-1;r), k=r,\ldots, n-1; j=r+1,\ldots, n-1\}.
\end{split}
\]
\end{prop}

The above formula provides an algorithm for computing $\mathcal{A}(n;r)$ recursively, and thus obtaining polynomials $Z_{n;r}$ given in \eqref{eq:Z_I;j}.
Figures \ref{fig:A(2;1)-to-A(3;1)} through \ref{fig:Z_234;1-c}  demonstrate the algorithm, by computing $\mathcal{A}(3;1)$ and $\mathcal{A}(4;1)$ and yield the arrow polynomial $Z_{1,2,3,4;1}=Z_{234;1}$. Figure \ref{fig:A(2;1)-to-A(3;1)} shows how the set $\mathcal{A}(3;1)=\{A,B,C\}$ is obtained from expansions of the diagram $e\in \mathcal{A}(2;1)$, where $A=e\prec_2 e$, $B=e\prec_1 e$, $C=e\prec_2 \bar{e}$. Figure \ref{fig:A(2;1)-to-A(3;1)} shows the corresponding arrow polynomial $Z_{23;1}=A+B-C$, see Equation \eqref{eq:tree-invariants-n=3}.
\begin{figure}[!ht] 
	\begin{overpic}[width=.7\textwidth]{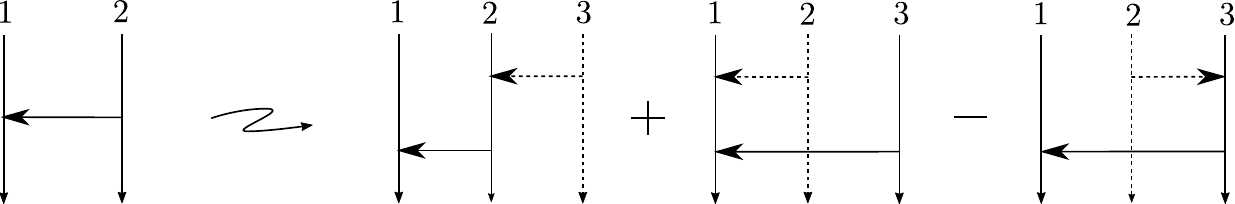}
	\end{overpic}
	\caption{Obtaining $\mathcal{A}(3;1)=\{A,B,C\}$ from $\mathcal{A}(2;1)=\{e\}$. }\label{fig:A(2;1)-to-A(3;1)}
\end{figure}	
%
\begin{figure}[!ht] 
	\begin{overpic}[width=\textwidth]{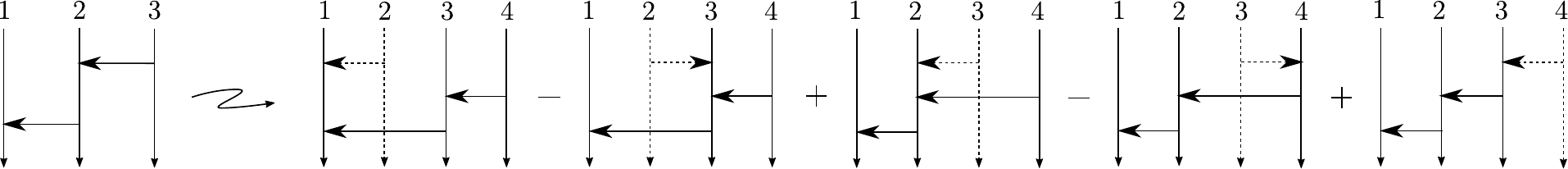}
	\end{overpic}
	\caption{Applying elementary expansions $\prec e$ and $\prec \bar{e}$ to $A\in \mathcal{A}(3;1)$.}\label{fig:Z_234;1-a}
\end{figure}
\begin{figure}[!ht] 
	\begin{overpic}[width=\textwidth]{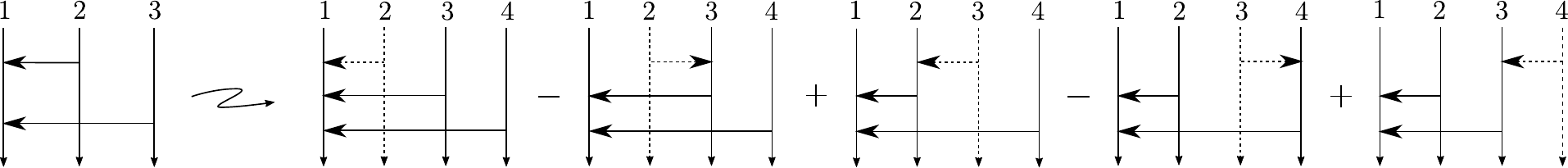}
	\end{overpic}
	\caption{Applying elementary expansions $\prec e$ and $\prec \bar{e}$ to $B\in \mathcal{A}(3;1)$.}\label{fig:Z_234;1-b}
\end{figure}	
\begin{figure}[!ht] 
	\begin{overpic}[width=\textwidth]{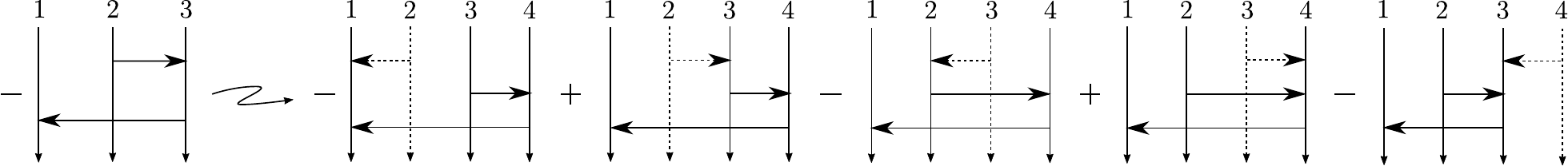}
	\end{overpic}
	\caption{Applying elementary expansions $\prec e$ and $\prec \bar{e}$ to $C\in \mathcal{A}(3;1)$.}\label{fig:Z_234;1-c}
\end{figure}	

Expanding the elements of $\mathcal{A}(3;1)$ and eliminating duplicates  yields $\mathcal{A}(4;1)$. Observe that the second to last term in Figure \ref{fig:Z_234;1-b} expansion is isomorphic to the first term of expansion in Figure \ref{fig:Z_234;1-c}, and the last  term in Figure \ref{fig:Z_234;1-b} expansion is isomorphic to the first term in Figure \ref{fig:Z_234;1-a}. Removing these duplicates yields the following arrow polynomial formula for the $Z_{2 3 4;1}$ invariant, which agrees with the formula obtained by Kravchenko and Polyak in \cite[p. 307]{Kravchenko-Polyak:2011},
\begin{equation}\label{eq:tree-invariants-n=4}
\begin{split}
Z_{2 3 4;1} (\ell) = \Bigl\langle&\vvcenteredinclude{0.3}{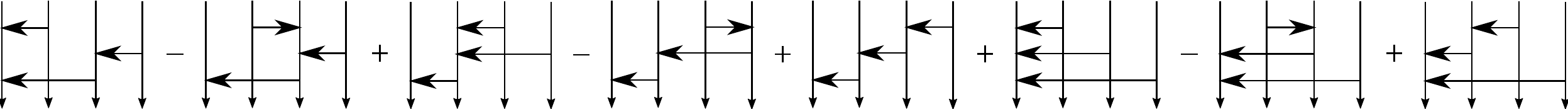} \\
&\vvcenteredinclude{0.3}{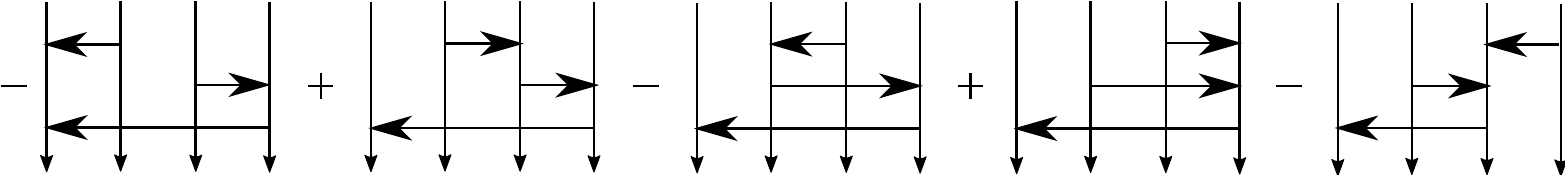}, G\Bigr\rangle .
\end{split}
\end{equation}
The next lemma tells us that, with appropriate choices of indexes, the operations $\prec_k$ obey a certain natural associativity formula (see \eqref{eq:I(P;Z)-I(Q;Z)} and \eqref{eq:i(j,B_i;A)} for the index notation). 
\begin{lem}\label{lem:shuffle}
 Let $P\in \mathcal{A}(p;v)$, $Q\in \mathcal{A}(q;w)$, $R\in \mathcal{A}(r;l)$. For $i\in [p]$ and $j\in [p+q-1]$, we have
 \begin{equation}\label{eq:shuffle}
  B:=(P\prec_i Q)\prec_j R =\begin{cases}
  B':=(P\prec_{j'} R)\prec_{i'} Q, & \text{if}\ j\in I(P;A), j\not\in I(Q;A),\\
  B'':=P\prec_{i''} (Q\prec_{j''} R), & \text{if}\ j\in I(Q;A),
  \end{cases}
 \end{equation}
 where $A=P\prec_i Q$, and the indexes $i$, $j$, $i'$, $j'$ and $i''$, $j''$ are related by 
 \[
  \begin{split}
  i(j,P;B)=i(j',P;B'), & \qquad i(i,P;B)=i(i',P;B'),\\
  i=i'', &\qquad  i(j,A;B)=i(j'',Q;B''),
  \end{split}
 \]
 respectively for the first and second identity in \eqref{eq:shuffle}.
\end{lem}
\begin{proof} The obvious graph isomorphisms are pictured in Figure \ref{fig:shuffle1} and Figure \ref{fig:shuffle2}.\end{proof}
  \begin{figure}[ht]
  	\includegraphics[width=.55\textwidth]{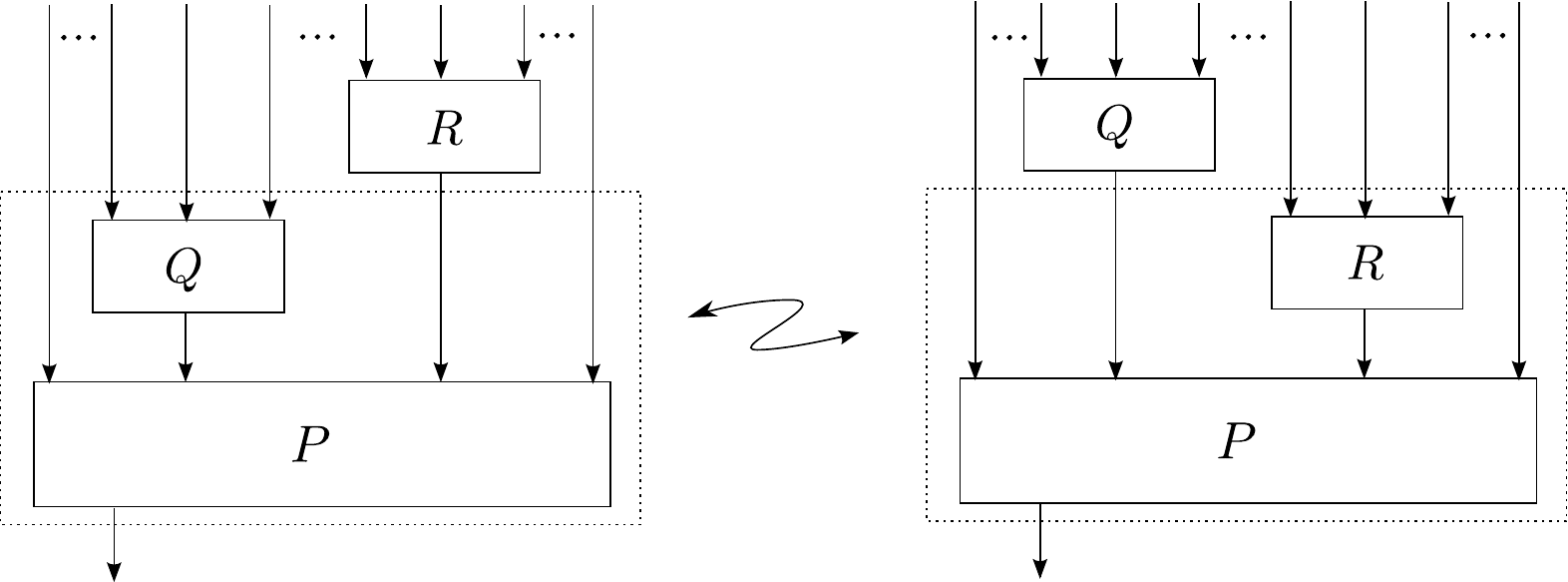}
  	\caption{Graph isomorphism $B$ and $B'$  in Equation \eqref{eq:shuffle}.}\label{fig:shuffle1}
  \end{figure}
      \begin{figure}[ht]
      	\includegraphics[width=.55\textwidth]{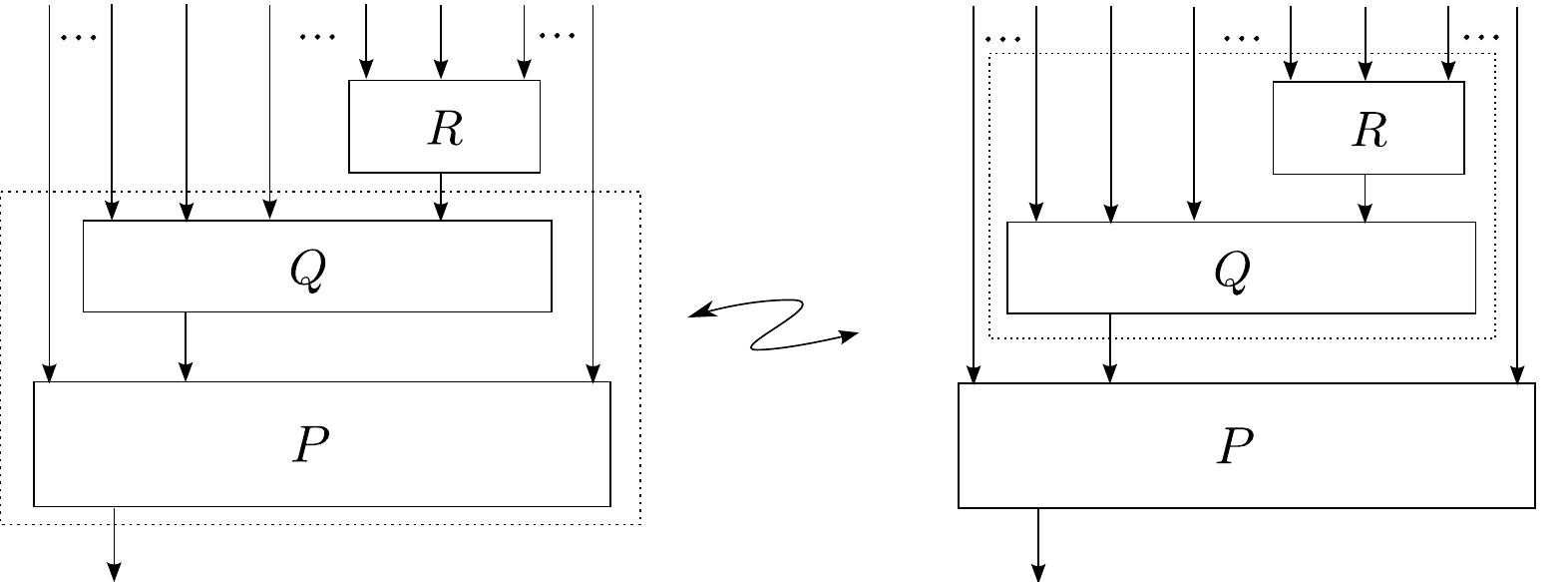}
      	\caption{Graph isomorphism  of $B$ and $B''$ in Equation \eqref{eq:shuffle}.}\label{fig:shuffle2}
      \end{figure}
    \begin{figure}[ht]
    	\includegraphics[width=.3\textwidth]{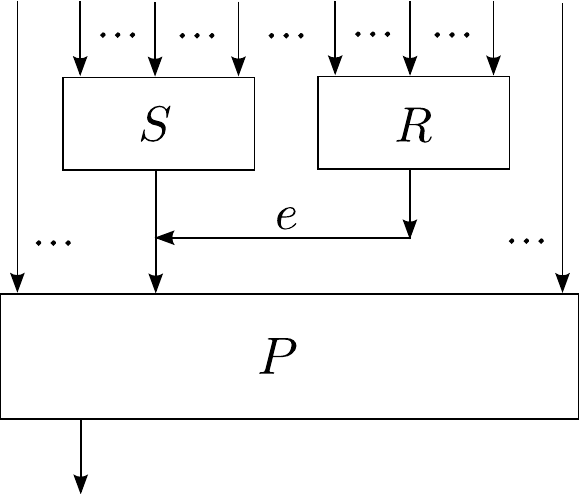}\qquad\quad\quad	\includegraphics[width=.3\textwidth]{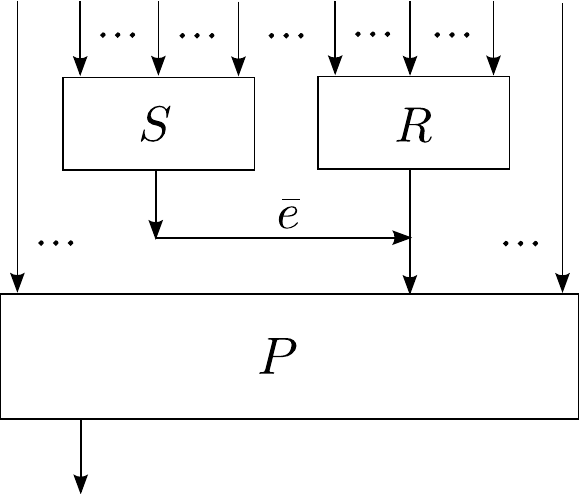}
    	\caption{Decomposition of $A\in \mathcal{A}(n;t)$ with respect to the arrow $\alpha$ corresponding to $e$ (left) or $\bar{e}$ (right).}\label{fig:arrow-decomp}
    \end{figure}
    \no The next lemma shows that any diagram in $\mathcal{A}(n;t)$ can be decomposed with respect to a given arrow in $A$ as pictured in Figure \ref{fig:arrow-decomp} which  is a natural consequence of Lemma \ref{lem:A(n;r)-recursive}.
\begin{lem}\label{lem:P,Q-decomp}
	For any given $A\in \mathcal{A}(n;t)$ and $\alpha\sim (i,j)\in A$, there exist $P\in \mathcal{A}(p;u)$, $R\in \mathcal{A}(r; v)$, and $S\in \mathcal{A}(s;w)$, $n=p+r+s-1$ so that
	\begin{align}
		\label{eq:A-decomp-e} A =P\prec_{v}  Q, & \quad\text{for}\quad Q=(e\prec_2 R)\prec_1 S,\quad\text{if}\ i<j,\\
		\label{eq:A-decomp-bar-e} & \quad\text{and}\quad Q=(\bar{e}\prec_2 R)\prec_1 S
		,\quad\text{if}\ i>j,\ 
	\end{align}
	where $v=i(i,A;P)$ and  $e$ (or $\bar{e}$ respectively) corresponds to $\alpha$ in $A$. 
\end{lem}
\begin{proof}
		\no The proofs of \eqref{eq:A-decomp-e} and \eqref{eq:A-decomp-bar-e} are analogous and follow by induction with respect to $n$, thus we provide details only in the case \eqref{eq:A-decomp-e}. 
		The base case $n=2$ follows immediately by choosing $P=R=S=\varnothing$. For the inductive step, suppose \eqref{eq:A-decomp-e} is valid for $n$, and pick any $A\in \mathcal{A}(n+1;t)$. By Lemma \ref{lem:A(n;r)-recursive}, one of the following holds for some $k\in [n-1]$:
		\begin{equation}\label{eq:1-2-cases}
		\begin{split}
		1^\circ.\ & A=A'\prec_k e,\\
		2^\circ.\ & A=A'\prec_k \bar{e},\quad \text{for}\quad k>1.
		\end{split}
		\end{equation}
		In the case $1^\circ$; if $\alpha=e$, $\alpha\sim (i,j)$ then $i=k$, $j=k+1$ and  \eqref{eq:A-decomp-e}is obtained by setting $P=A'$, $Q=e$, $R=\varnothing$, $S=\varnothing$.
		If $\alpha$ is different from $e$ in \eqref{eq:1-2-cases}, there must be $\alpha'\in A'$, which becomes $\alpha$ in $A$, after stacking either $e$ or $\bar{e}$, according to \eqref{eq:1-2-cases}. 
		Let $(i',j')$ be coordinates of $\alpha'$, then $i'$ equals either $i-1$ or $i$.
		By the inductive hypothesis applied to $A'$ with $\alpha'$ we obtain
		\begin{equation}\label{eq:A'}
		A' =P'\prec_{v'} Q',  \quad Q'=(e\prec_2 R')\prec_1 S',
		\end{equation}
		for certain $P'$, $Q'$, $R'$, $S'$ and $v'=i(i',A';P')$ where $e$  in \eqref{eq:A'} corresponds to $\alpha'$. Using Lemma \ref{lem:shuffle} and \eqref{eq:A'}, we have the following obvious subcases; for $x=e$ or $x=\bar{e}$ as in \eqref{eq:1-2-cases}:
		\begin{itemize}
		\item[$(a)$] for $k\in I(P';A')$ and $k\not\in I(Q';A')$ we obtain \eqref{eq:A-decomp-e} by defining 
		\[
		 P=P'\prec_{w} x,\ Q=Q',\quad w=i(k,A';P').
		\]
		\item[$(b)$] for $k\in I(S';A')$  we obtain \eqref{eq:A-decomp-e} by defining $P=P'$, 
		\[
		 R=R',\  S=S'\prec_w x,\quad w=i(k,A';S').
		 \] 
		\item[$(c)$] for $k\in I(R';A')$  we obtain \eqref{eq:A-decomp-e} by defining 
		\[
		 P=P',\ S=S',\ R=R'\prec_w x,\quad w=i(k,A';R'). \qedhere
		 \]  
\end{itemize}
\end{proof}

\begin{rem}[Operad structure of $\mathcal{A}$ and $Dias$]\label{rem:operad}
	Lemmas of this section can be formulated in the language of operads \cite{Loday-Vallette:2012, May:1972}. In particular, for the sequence of sets\footnote{$\mathcal{A}(n)=\bigcup_j \mathcal{A}(n;j)$.} $\mathcal{A}=(\mathcal{A}(n))_{n\in \mathbb{N}}$, the stacking operations of \eqref{eq:stack_k}:
	\[
	\prec_k:\mathcal{A}(n) \times \mathcal{A}(m)\longrightarrow \mathcal{A}(m+n-1),\qquad (P,Q) \longrightarrow P\prec_k Q,
	\] 
	give {\em partial compositions} defined in \cite[p. 115]{Loday-Vallette:2012} which, by Lemma \ref{lem:shuffle}, satisfiy relations (I) and (II) of \cite[p. 116]{Loday-Vallette:2012}. In turn, Proposition 5.3.8 of \cite[p. 117]{Loday-Vallette:2012} implies that $\mathcal{A}$ is an operad.  Lemma \ref{lem:A(n;r)-recursive} shows that the operad $\mathcal{A}$ is binary, i.e. generated by the two elements $e$ and $\bar{e}$.
	
	 Taking one step further, it appears that the partial compositions $\prec_k$ descend via the quotient map $\mathcal{A}(n)\longrightarrow Dias(n)$ to $Dias=(Dias(n))_{n\in\mathbb{N}}$, i.e.
	\[
	 	\prec_k:Dias(n) \times Dias(m)\longrightarrow Dias(m+n-1),
	\]
	which in turn would imply that $Dias$ is binary.	We leave the details of proofs for these statements to the reader.
\end{rem}

\section{Proof of Main Theorem}\label{S:proof}
Recall the statement of  Main Theorem; Given a based $n$--component link  $L$ in $\R^3$ and  its Gauss diagram $G_L$, the following quantity defines a homotopy invariant of $L$:
\begin{equation}\label{eq2:bar-Z_I;j}
\overline{Z}_{I;j}(L)=\langle Z_{I;j}, G_L\rangle \mod \Delta_Z(I;j),
\end{equation}  
for $I=\{i_1,i_2,\ldots, i_r\}$, $1\leq i_1 <i_2 <\ldots <i_r \leq n$ and $j=i_k$, $1\leq k\leq r$, where
\begin{equation}\label{eq2:Delta_Z(I;j)}
\Delta_Z(I;j)=\gcd\{\langle Z_{J;k}, G_L\rangle\ |\ J\subsetneq I,\ k\in J\}.
\end{equation}
 Apart from the presence of the indeterminacy $\Delta_Z(I;j)$ in \eqref{eq2:bar-Z_I;j} the main point of the above formula is that the tree polynomial $Z_{I;j}$ is evaluated on a Gauss diagram of a link $G_L$ rather than a string link $G_\ell$ as in Theorem \ref{thm:Z_I;j-invariants}.

The proof of Main Theorem is divided into two parts; In the first part we analyze the difference $\langle Z_{I;j},G_L\rangle-\langle Z_{I;j},G'_L\rangle$, where $G'_L$ is a Gauss diagram of $L$ obtained after  moving the basepoint of a component past an over/undercrossing. A similar basepoint change argument can be found in the work of  {\"O}stlund, \cite{Ostlund:2004}, for $n=3$, but rather than using {\"O}stlund's {\em diagram fragments} we perform direct computations with tree diagrams and their suitable tree decompositions as stated in Lemma \ref{lem:P,Q-decomp}.
In the second part of the proof, we show invariance under the Reidemeister moves, where the argument is analogous as the one by Kravchenko and Polyak in  Theorem \ref{thm:Z_j} of \cite{Kravchenko-Polyak:2011} and is included mainly for completeness of the exposition.
\subsection{Basepoint change}
 Let us review our basic notation first; let $G=G_L$ be a Gauss diagram of the $n$--component based link $L$. For convenience we will always visualize $G$ as a string rather than a circular Gauss diagram; see Figure \ref{fig:gauss-d}.  
 A {\em{embedding}} of a tree diagram $A\in \mathcal{A}(m;k)$ in $G$ is a graph embedding $\phi:A\longrightarrow G$ of $A$ into $G$ mapping components of $A$ to the strings  of $G$, preserving the endpoints and arrow orientations. Recall that arrows of $A$ are denoted by the greek letters $\alpha$, $\beta$,\ldots , and arrows of $G$ are denoted by lowercase letters: $g$, $h$,\ldots . Further, the pairing $\langle A, G\rangle$ is given by the sum 
 \begin{equation}\label{eq2:<A,G>}
 	\langle A, G\rangle = \sum_{\phi: A \longrightarrow G} \sign(\phi),\qquad 	\text{where}\qquad \sign(\phi)=\prod_{\alpha \in A} \sign(\phi(\alpha)),
 \end{equation}
 \no taken over all embeddings of $A$ in $G$. The following notation for a partial sum  will be used frequently
 \begin{equation}\label{eq:<A,G>_cond}
 \langle A, G\rangle_{\textbf{cond.}} = \sum_{\substack{\phi: A \longrightarrow G;\\ \phi\ \text{satisfies}\ \textbf{cond.}} } \sign(\phi),
 \end{equation}
 where the sum is only over those embeddings which satisfy a given condition: {\bf cond.}. For instance  given an arrow $g\in G$ and an arrow $\alpha\in A$, we write 
 \begin{equation}\label{eq:<A,G>_arr}
 	\langle A, G\rangle_{\alpha\mapsto g} = \sum_{\substack{\phi: A \rightarrow G;\\ \phi(\alpha)=g}} \sign(\phi),\qquad\quad \langle A, G\rangle_{\alpha\not\mapsto g} = \sum_{\substack{\phi: A \rightarrow G;\\ \phi(\alpha)\neq g}} \sign(\phi).
 \end{equation}
Note that, for any $\alpha$ and $g$,
 \begin{equation}\label{eq:alpha->g}
 \langle A, G\rangle=\langle A, G\rangle_{\alpha\mapsto g}+\langle A, G\rangle_{\alpha\not\mapsto g}.
 \end{equation}
 If a Gauss diagram $G$ has $n$ components, $A\in \mathcal{A}(m;q)$, $m\leq n$, and  $I=\{i_1,\ldots, i_m\}$, $1\leq i_1<i_2<\ldots<i_m\leq n$, we denote by $G(I)$ the subdiagram of $G$ obtained by removing the components (together
  with the adjacent arrows) of $G$ which are not in $I$, and define  
 \begin{equation}\label{eq:<A,G(I)>}
 	\langle A, G(I)\rangle = \sum_{\phi: A \rightarrow G;\ \phi(A)\subset G(I)} \sign(\phi).
 \end{equation}
From the definitions presented in Section \ref{S:intro}, $\langle A, G(I)\rangle$, equals $\langle A',G\rangle$ where $A'\in  \mathcal{A}(I;i_q)$ is obtained from $A$ by simply indexing the leaves and the trunk by $(I;i_q)$.

Each step in the following analysis is pictured in Figures \ref{fig:basepoint-pass-i} through \ref{fig:basepoint-pass-iv}, where  $G$ (left) and $G'$ (right) differ by a basepoint move, i.e. the arrow $g\sim (i,j)$ in $G$ becomes $g'$ in $G'$ and the dashed part of the diagrams is common for both $G$ and $G'$. We consider four cases $(\iota)$--$(\iota v)$ as the basepoint of the $i$th or $j$th string passes through  the arrowhead/tail of  $g$. The basepoint passes are denoted by $(a)$ and $(b)$ in each of the Figures \ref{fig:basepoint-pass-i}--\ref{fig:basepoint-pass-iv}. The following result characterizes the difference $\langle Z_{n;t},G\rangle-\langle Z_{n;t}, G'\rangle$, for $1\leq t\leq n$.
\begin{lem}\label{lem:basepoint-change}
Let $G$ be a  Gauss diagram of a based closed link $L$, and consider a Gauss diagram $G'$ obtained from $G$ via a moving a basepoint past a crossing in each case: $(\iota)$--$(\iota v)$. Then, we have the following identity
	\begin{equation}\label{eq:z(G)-z(G')}
	 \langle Z_{n;t},G\rangle-\langle Z_{n;t}, G'\rangle=\sum^n_{j=1}\sum_{J\in \mathcal{I}_j} a_{J;j} \langle Z_{J;j}, G\rangle,
	 \end{equation}
for some integer coefficients $a_{J;j}$ and 
where the index sets $\mathcal{I}_j$ are given as follows 
\begin{equation}\label{eq2:I_j}
\begin{split}
\mathcal{I}_t & =\bigl\{[1,i]\cup [k,n]\ |\ i=1,\ldots, n-1; i+1<k\leq n+1; t\in [1,i]\text{ or } t\in [k,n]\bigr\}\\
&\qquad \cup \bigl\{[1,k]\cup [i,n]\ |\ 0\leq k<i-1;i\leq n; t\in [1,k]\text{ or } t\in [i,n]\bigr\},\\
\mathcal{I}_j & =\bigl\{[k,\ldots,m]\ |\ 1\leq k\leq j;j\leq m\leq n; k<m\bigr\},\quad j\neq t.
\end{split}	 
\end{equation}
where $[a,b]=\varnothing$ for $a>b$, and $[a,a]=\{a\}$.
\end{lem}
\begin{proof}
\no  For brevity, let $z(\,\cdot\,)$ denote $\langle Z_{n;t},\,\cdot\,\rangle$. Our basic ingredient in the computation of $z(G)-z(G')$ in \eqref{eq:z(G)-z(G')} is a computation of the difference
\begin{equation}\label{eq:AG-AG'}
\langle A, G\rangle-\langle A, G'\rangle\quad \text{for any}\quad A\in \mathcal{A}(n;t).
\end{equation}
 Since the diagrams $G$ and $G'$ differ only by the location of the arrow $g$ ($g\sim (i,j)$) there are two basic cases; first, there  exists an arrow $\alpha\in A$, $\alpha\sim (i,j)$, potentially matching $g$ under an embedding of $A$ in $G$. In  the second case there is no arrow in $A$ with coordinates $(i,j)$. 
In the first case, we may express \eqref{eq:AG-AG'}, using \eqref{eq:alpha->g}, as
\begin{equation}\label{eq:AG-diff}
\langle A, G\rangle - \langle A, G'\rangle  = \bigl(\langle A, G\rangle_{\alpha\mapsto g}-\langle A, G'\rangle_{\alpha\mapsto g'}\bigr)+\bigl(\langle A, G\rangle_{\alpha\not\mapsto g}-\langle A, G'\rangle_{\alpha\not\mapsto g'}\bigr).
\end{equation}
 Note that  the second term of \eqref{eq:AG-diff} vanishes, i.e.
\begin{equation}\label{eq:A,G=A,G'}
	\langle A, G\rangle_{\alpha\not\mapsto g}-\langle A, G'\rangle_{\alpha\not\mapsto g'}=0.
\end{equation}
Indeed, for any embedding $\phi:A\longmapsto G$, such that $\phi(\alpha)\neq g$ (i.e. $\phi(A)\subset G-\{g\}$). Since $G-\{g\}$ and $G'-\{g'\}$ are identical we can define $\phi':A\longmapsto G'$ as a composition of $\phi$ and the inclusion $G'-\{g'\}\subset G'$. Because $\sign(\phi)=\sign(\phi')$  we conclude $\langle A, G\rangle_{\alpha\not\mapsto g}=\langle A, G'\rangle_{\alpha\not\mapsto g'}$ which proves \eqref{eq:A,G=A,G'}. 
If a diagram $A$ has no arrow $\alpha$ matching $g$ in $G$, the first term in the sum \eqref{eq:AG-diff} vanishes, and $\langle A, G\rangle - \langle A, G'\rangle = \langle A, G\rangle_{\alpha\not\mapsto g}-\langle A, G'\rangle_{\alpha\not\mapsto g'}$. Then the same argument as for \eqref{eq:A,G=A,G'} implies $\langle A, G\rangle - \langle A, G'\rangle=0$.  Therefore, we obtain  
\begin{equation}\label{eq:AG-diff-alpha}
\langle A, G\rangle - \langle A, G'\rangle  = \langle A, G\rangle_{\alpha\mapsto g}-\langle A, G'\rangle_{\alpha\mapsto g'}
\end{equation}

\no Next, we analyze in detail Cases $(\iota)$--$(\iota v)$ illustrated in Figures  \ref{fig:basepoint-pass-i}--\ref{fig:basepoint-pass-iv}.
\smallskip
	
\no {\bf Case $(\iota)$:} Suppose $g\in G$, $g\sim (i,j)$, $i<j$ has its arrowhead closest to the basepoint along the $i$th string as pictured in Figure \ref{fig:basepoint-pass-i}, and $G'$ is obtained from $G$ by applying move $(a)$ (or equivalently $G$ is obtained from $G'$ via the move $(b)$).  
\begin{figure}[ht]
	\includegraphics[width=.8\textwidth]{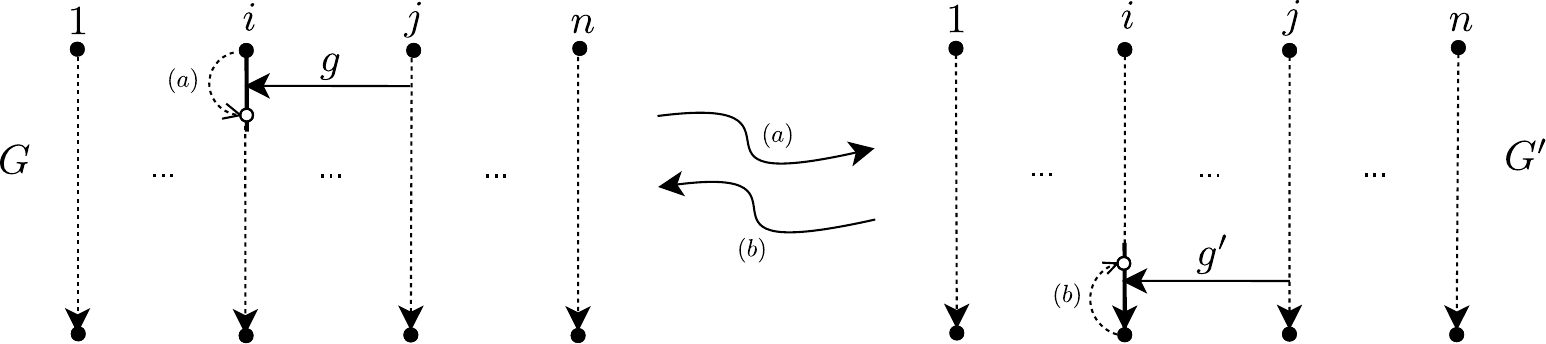}
	\caption{Case $(\iota)$: Basepoint moving pass the arrowhead of $g\sim (i,j)$, $i<j$.}\label{fig:basepoint-pass-i}
\end{figure}
\begin{figure}[ht]
	\includegraphics[width=.35\textwidth]{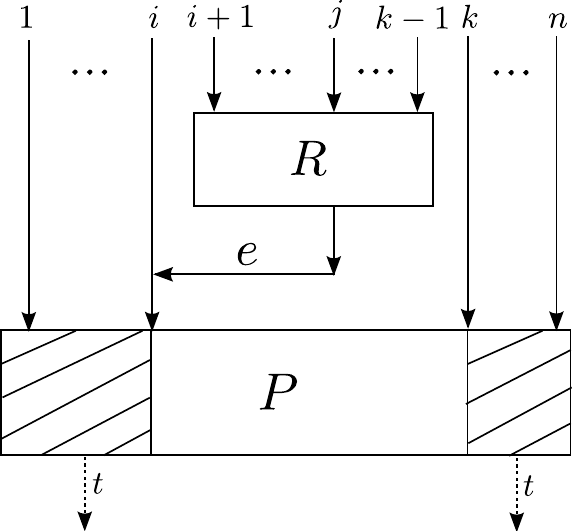}\qquad\qquad \includegraphics[width=.35\textwidth]{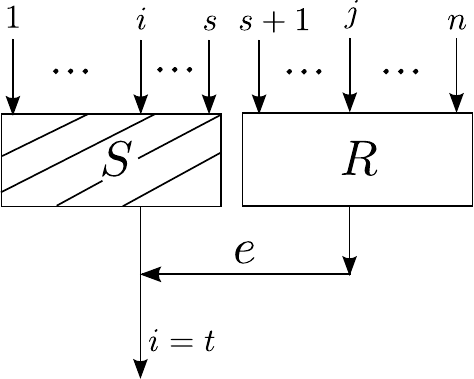}
	\caption{Diagrams in  $(\iota.b)$ and $(i.c)$, where $e$ corresponds to $\alpha$, the dashed vertical arrows, in the left diagram, indicate a possible location for the trunk of $A$ indexed by $t$.}\label{fig:i_b-i_c}
\end{figure}
The following cases define certain disjoint subsets in $\mathcal{A}(n;t)$ of tree diagrams which may yield a nonzero right hand side of \eqref{eq:AG-diff-alpha}. Roughly speaking, Case $(\iota.b)$ concerns  the diagrams which have an arrow  $\alpha\sim (i,j)$ with $h(\alpha)$ at the top of component $i$, such diagrams possibly yield $\langle A, G\rangle_{\alpha\mapsto g}\neq 0$ (Figure \ref{fig:i_b-i_c}(left)). Case $(\iota.c)$ concerns diagrams having $\alpha\sim (i,j)$ with $h(\alpha)$ at the bottom of the $i$--component, such diagrams possibly contribute to $\langle A, G'\rangle_{\alpha\mapsto g'}$ (Figure \ref{fig:i_b-i_c}(right)). Case $(\iota.a)$ is a special subcase of $(\iota.c)$ and concerns diagrams with a single $\alpha\sim (i,j)$ on the $i$--component, this can only happen for $t=i=1$ (by ({\bf d1})--({\bf d4})).

\no  {\bf Case $(\iota.a)$}, (diagrams possibly contributing to both $\langle A, G\rangle_{\alpha\mapsto g}$ and $\langle A, G'\rangle_{\alpha\mapsto g'}$ in \eqref{eq:AG-diff-alpha}): Define
\begin{equation}\label{eq:A_(i.a)}
\begin{split}
 \mathcal{A}_{(\iota.a)}(j) & =\{A\in \mathcal{A}(n;1)\ |\ \text{exists $\alpha\in A$, such that}\ \alpha\sim (1,j),\ \text{and $\alpha$ is the only arrow}\\
 &\qquad\qquad\qquad\qquad \text{with its head on the first component of $A$}\}
 \end{split}
\end{equation}
(Figure \ref{fig:i_b-i_c}(left) with $S=\varnothing$ shows the general form of such diagrams). Then given an embedding $\phi:A\longmapsto G$, $\phi(\alpha)=g$, we may define $\phi':A\longmapsto G'$ to differ from $\phi$ only by assigning $g'$ to $\alpha$, i.e. $\phi'(\alpha)=g'$.  Because $\sign(\phi)=\sign(\phi')$ we obtain 
$\langle A,G\rangle_{\alpha\mapsto g}=\langle A,G'\rangle_{\alpha\mapsto g'}$, yielding zero on the left hand side of \eqref{eq:AG-diff-alpha}, and as a consequence;
\begin{equation}\label{eq:z(G)-z(G')-i_a}
 \sum_{A\in \mathcal{A}_{(\iota.a)}(j)} \sign(A)\bigl( \langle A, G\rangle-\langle A, G'\rangle \bigr)=0.
\end{equation} 
 
\no  {\bf Case $(\iota.b)$} (diagrams possibly contributing to $\langle A, G\rangle_{\alpha\mapsto g}$ in \eqref{eq:AG-diff-alpha}):  
Consider diagrams $A\in \mathcal{A}(n;k)$ with an arrow  $\alpha\sim (i,j)$, 
  which admit the following decomposition with respect to $\alpha$ (Figure \ref{fig:i_b-i_c}(left) and Equation \eqref{eq:A-decomp-e} of Lemma \ref{lem:P,Q-decomp}, with $S=\varnothing$, $P\neq\varnothing$) 
\begin{equation}\label{eq:A-decomp-i_b}
 A=P\prec_i Q,\qquad Q=e\prec_2 R,
\end{equation}
 where $\alpha$ corresponds to $e$. Note, that  there is no other arrowhead above the arrowhead of $\alpha$ along the $i$--component of $A$, hence $l(R)=j-i-1$ (see notation in \eqref{eq:r-l-trunk}) and because $P\neq \varnothing$, $A$ cannot be in $\mathcal{A}_{(\iota.a)}(j)$.
 
  For a fixed $I\subset [n]$ and $J=([n]-I)\cup \{i\}$, denote a set of  diagrams $A\in \mathcal{A}(n;t)$ admitting the decomposition \eqref{eq:A-decomp-i_b} with $I=I(P;A)$ and $J=I(Q;A)$, as 
  \[
  \mathcal{A}_{(\iota.b)}=\mathcal{A}_{(\iota.b)}(i,j;I;t).
  \]
   Clearly, not every $I\subset [n]$ yields a nonempty  $\mathcal{A}_{(\iota.b)}(i,j;I;t)$, we must have $I$ and $J$ as given by \eqref{eq:I(P;Z)-I(Q;Z)}; with $k=i$ and $l(R)=j-i-1$. For further convenience, we define
\[
\mathcal{I}_{(\iota.b)}(i,j)=\{I\subset [n]\ |\ \mathcal{A}_{(\iota.b)}(i,j;I;t)\neq \varnothing\}.
\]
Using \eqref{eq:I(P;Z)-I(Q;Z)} and \eqref{eq:A-decomp-i_b} it follows that $I(P;A)=[1,i]\cup [k,n]$, for an appropriate $k>j$ (see Figure \ref{fig:i_b-i_c}, the ``shaded'' part of $P$) and therefore
\begin{equation}\label{eq:I_(i.b)} 
\mathcal{I}_{(\iota.b)}(i,j)=\{[1,i]\cup [k,n]\ |\ k>j\}.
\end{equation}
For $A\in \mathcal{A}_{(\iota.b)}$, any embedding
   $\phi:A\longmapsto G$, $\phi(\alpha)=g$ has two restrictions
\[
\xi_I=\phi|_{P}:P\longmapsto G(I),\qquad \psi_J=\phi|_{Q}:Q\longmapsto G(J),\quad \psi_J(\alpha)=g,\quad J=([n]-I)\cup \{i\}.
\]
Conversely, consider any embedding 
$\xi_I:P\longmapsto G(I)\subset G$, and an arrow $\beta\in P$. If the image arrow: $\xi_I(\beta)$ has its head/tail on the $i$th string of $G$ it must be below the head of $g$ and  since $I\cap J=\{i\}$,  
any pair of embeddings $\xi_I:P\longmapsto G(I)$,  $\psi_J:Q\longmapsto G(J)$, $\psi_J(\alpha)=g$, yields the ``joint'' embedding
\[
\phi=\xi_I\sqcup \psi_J:A\longmapsto G,\quad A=P\prec_i Q.
\]
Since $\sign(\phi)=\sign(\xi_I)\sign(\psi_J)$, we have
\begin{equation}\label{eq:AG_alpha->g}
\begin{split}
\langle A,G\rangle_{\alpha\mapsto g} & =\sum_{\substack{\phi:A\longmapsto G,\\ \phi(\alpha)=g}} \sign(\phi)=\sum_{\substack{\xi_I\sqcup \psi_J:A\longmapsto G,\\ \psi_J(\alpha)=g}} \sign(\xi_I)\sign(\psi_J)\\
& 
= \Bigl(\sum_{\xi_I:P\longmapsto G(I)} \sign(\xi_I)\Bigr)\Bigl(\sum_{\substack{\psi_J:Q\longmapsto G(J),\\ \psi_J(\alpha)=g}} \sign(\psi_J)\Bigr) =\langle P, G(I)\rangle\, \langle Q, G(J)\rangle_{\alpha\mapsto g} 
\end{split}
\end{equation}
(using the notation in \eqref{eq:<A,G(I)>}). On the other hand, since $\alpha$ has a top arrowhead along the $i$--component of $A$, there must be at least one other arrowhead on the $i$--component. Axiom ({\bf d4}) implies that there can be no embedding $\phi:A\longmapsto G'$, such that $\phi(\alpha)=g'$ because there is no arrowhead/tail below the head of $g'$ along the $i$th string of $G'$ (Figure \ref{fig:basepoint-pass-i}(right)). In turn we obtain $\langle A,G'\rangle_{\alpha\mapsto g'}=0$, and from \eqref{eq:AG-diff-alpha}, \eqref{eq:AG_alpha->g}, we may compute, for a fixed $I\in \mathcal{I}_{(\iota.b)}(i,j)$, $J=([n]-I)\cup\{i\}$:
\begin{equation}\label{eq:z-z'-i_b}
\begin{split}
& \sum_{A\in \mathcal{A}_{(\iota.b)}(i,j;I;t)}                                                                                                                                                                                                                                                                                                                                                                                                                                                                                                                                                                                                                                    \sign(A)\bigl(\langle A, G\rangle - \langle A, G'\rangle)  = \sum_{A\in \mathcal{A}_{(\iota.b)}(i,j;I;t)}  \sign(A)\langle A,G\rangle_{\alpha\mapsto g}\\
&\qquad=\sum_{\substack{P\in \mathcal{A}(p;r);\\ p=|I|}} \sum_{\substack{Q\in \mathcal{A}(q;1)\\ Q=e\prec_2 R; q=|J|}} \sign(P)\sign(Q)\langle P, G(I)\rangle\, \langle Q, G(J)\rangle_{\alpha\mapsto g}.
 \end{split}
\end{equation}
where we used the sign identity \eqref{eq:sign-P-Q}, and $r$ is determined by \eqref{eq:P-Q-trunk} with $m=q$ and $t$ given in \eqref{eq:z(G)-z(G')}.
For any $I\in \mathcal{I}_{(\iota.b)}(i,j)$, Equation \eqref{eq:Z_I;j} implies  
\begin{equation}\label{eq:P-Z_{I;r}-i_b}
\sum_{\substack{P\in \mathcal{A}(p;r);\\ p=|I|}} \sign(P)\langle P, G(I)\rangle=Z_{I;t}(G).
\end{equation}
For this reason we will refer to $P$ as a {\em free factor} in the decomposition \eqref{eq:A-decomp-i_b}. Setting $b_I=\sum_{\substack{Q\in \mathcal{A}(q;1)\\ Q=e\prec_2 R; q=|J|}} \sign(Q)\langle Q, G(J)\rangle_{\alpha\mapsto g}$,  we obtain from \eqref{eq:z-z'-i_b} and \eqref{eq:P-Z_{I;r}-i_b}:
\begin{equation}\label{eq:z(G)-z(G')-i_b}
\sum_{I\in \mathcal{I}_{(\iota.b)}(i,j)}\sum_{A\in \mathcal{A}_{(\iota.b)}(i,j;I;t)} \sign(A)\bigl(\langle A, G\rangle - \langle A, G'\rangle)=\sum_{I\in \mathcal{I}_{(\iota.b)}(i,j)} b_I\, Z_{I;t}(G).
\end{equation} 
\no {\bf Case  $(\iota.c)$} (diagrams possibly contributing to  $\langle A, G'\rangle_{\alpha\mapsto g'}$ in \eqref{eq:AG-diff-alpha}): 
Consider diagrams $A\in \mathcal{A}(n;t)$ which have an arrow $\alpha$ with $\alpha\sim (i,j)$, $i=t$ and the arrowhead at the bottom of the $i$--component of $A$. By Lemma \ref{lem:P,Q-decomp}, such diagrams 
 admit the following decomposition with respect to $\alpha$ (Figure \ref{fig:i_b-i_c}(right) and Equation \eqref{eq:A-decomp-e} with $P=\varnothing$)
\begin{equation}\label{eq:A-decomp-iota.c}
A=U\prec_1 S,\qquad  U=e\prec_2 R,
\end{equation}
where $\alpha$ is an arrow in $e$. Note that $\alpha$ is a bottom arrow, i.e. there is no other arrowhead below the head of $\alpha$ along the trunk of $A$. Because $\alpha\sim (t,j)$ we must have $l(R)+r(S)=j-t-1$. As in the case $(\iota.b)$, for a given $I\subset[n]$ and $J=([n]-I)\cup \{i\}$, we define the set of tree diagrams in $\mathcal{A}(n;t)$ admitting  the above decomposition \eqref{eq:A-decomp-iota.c}, with $I(U;A)=J$,  $I(S;A)=I$ as 
 $\mathcal{A}_{(\iota.c)}(t,j;I)$
and $\mathcal{I}_{(\iota.c)}(t,j)=\{I\subset [n]\ |\ \mathcal{A}_{(\iota.c)}(t,j;I)\neq \varnothing\}$.
From \eqref{eq:I(P;Z)-I(Q;Z)} it follows that $I(S;A)=[1,k]$, for $t\leq k<j$, yielding
\begin{equation}\label{eq:I_(i.c)} 
\mathcal{I}_{(\iota.c)}(t,j)=\{[1,k]\ |\ t\leq k<j\}.
\end{equation}
Given $A\in \mathcal{A}_{(\iota.c)}(t,j;I)$, observe that, unless $S=\varnothing$ which is covered by Case $(i.a)$, there is no embedding $\phi:A\longmapsto G$ with $\phi(\alpha)=g$, since there is no arrowhead above $g$ in $G$ along the $t$'th string of $G$.  As a result \eqref{eq:AG-diff-alpha} simplifies to
$\langle A, G\rangle - \langle A, G'\rangle  = -\langle A, G'\rangle_{\alpha\mapsto g'}$. 
Using \eqref{eq:A-decomp-iota.c}, every embedding $\phi':A\longmapsto G'$, $\phi'(\alpha)=g'$ restricts to subdiagrams $U$ and $S$, yielding $\xi'_J:U\longmapsto G'(J)$, $J=I(U;A)$ and $\psi'_I:S\longmapsto G'(I)$, $\psi'_J(\alpha)=g'$, $I=I(S;A)$. Because there is no arrowhead/tail below the head of $g'$ in $G'$, every pair $\xi'_J:U\longmapsto G'(J)$, $\psi'_I:S\longmapsto G'(I)$, $\psi'_J(\alpha)=g'$, gives an embedding $\phi'=\xi'_J\sqcup \psi'_I: A\longmapsto G'$, $\phi'(\alpha)=g'$. An analogous computation as in \eqref{eq:AG_alpha->g}, shows that $S$ is a free factor of decomposition \eqref{eq:A-decomp-iota.c} and 
\begin{equation}\label{eq:AG'_alpha->g'_i}
\langle A,G'\rangle_{\alpha\mapsto g'}=\langle S, G'(I)\rangle \langle U, G'(J)\rangle_{\alpha\mapsto g'},\qquad J=I(U;A),\ I=I(S;A).
\end{equation}
As in \eqref{eq:z-z'-i_b}, given $I\in \mathcal{I}_{(\iota.c)}(t,j)$  we compute
\begin{equation}\label{eq:z(G)-z(G')-i_c}
\begin{split}
& \sum_{A\in \mathcal{A}_{(\iota.c)}(t,j;I)}                                                                                                                                                                                                                                                                                                                                                                                                                                                                                                                                                                                                                                    \sign(A)\bigl(\langle A, G\rangle - \langle A, G'\rangle\bigr)  = -\sum_{A\in \mathcal{A}_{(\iota.c)}(t,j;I)}  \sign(A)\langle A,G'\rangle_{\alpha\mapsto g'}\\
&\qquad \qquad =-\sum_{\substack{S\in \mathcal{A}(s;t);\\ s=|I|}} \sum_{\substack{U\in \mathcal{A}(u;1)\\ U=e\prec_2 R; u=|J|}} \sign(S)\sign(U)\langle S, G(I)\rangle\, \langle U, G(J)\rangle_{\alpha\mapsto g'}= c_I\, Z_{I;t}(G),
\end{split}
\end{equation}
where $c_I=\sum_{\substack{U\in \mathcal{A}(u;1)\\ U=e\prec_2 R; u=|J|}} \sign(U)\langle U, G(J)\rangle_{\alpha\mapsto g'}$. 

Coming back to the main argument in Case $(\iota)$; note that subsets $\mathcal{A}_{(\iota.a)}$,  $\mathcal{A}_{(\iota.b)}$ and $\mathcal{A}_{(\iota.c)}$ of $\mathcal{A}(n;t)$  are disjoint. If $A$ does not belong to their union, then either $1^\circ$: it has no arrow $\alpha\sim (i,j)$, or $2^\circ$: there exists $\alpha\in A$ such that $\alpha\sim (i,j)$, but $\alpha$ is neither the top nor bottom arrow of $A$. For such a diagram $A$, we easily obtain from \eqref{eq:AG-diff-alpha}:
$\langle A, G\rangle - \langle A, G'\rangle=0$. Consequently, in Case $(\iota)$, based on \eqref{eq:z(G)-z(G')-i_a}, \eqref{eq:z(G)-z(G')-i_b} and \eqref{eq:z(G)-z(G')-i_c} we compute
\begin{equation}\label{eq:z(G)-z(G')-case-i}
\begin{split}
 z(G)-z(G') & =\sum_{A\in \mathcal{A}(n;t)} \sign(A)\bigl(\langle A, G\rangle - \langle A, G'\rangle\bigr)\\
 &=\sum_{i,j}\sum_{I\in \mathcal{I}_{(\iota.b)}(i,j)}\sum_{A\in \mathcal{A}_{(\iota.b)}(i,j;I;t)}   \sign(A)\langle A,G\rangle_{\alpha\mapsto g}\\
 & \qquad -\sum_{j}\sum_{I\in \mathcal{I}_{(\iota.c)}(t,j)}\sum_{A\in \mathcal{A}_{(\iota.c)}(t,j;I)}   \sign(A)\langle A,G'\rangle_{\alpha\mapsto g'}\\
 &\ =\sum_{i,j}\sum_{I\in \mathcal{I}_{(\iota.b)}(i,j)} b_I\, Z_{I;t}(G)-\sum_{j}\sum_{I\in \mathcal{I}_{(\iota.c)}(t,j)} c_I\, Z_{I;t}(G).
\end{split} 
\end{equation}
Proving Lemma \ref{lem:basepoint-change} in Case $(\iota)$, where the coefficients $a_{I;j}$ ought to be chosen as $b_I$ or $c_I$ above.

The proofs of the remaining cases have analogous steps,  in the following paragraphs we limit the amount of details  giving the most relevant parts and emphasizing the differences.

\no {\bf Case $(\iota\iota)$:} Suppose $g=(i,j)$ in $G$, $i>j$ has its arrowhead closest to the basepoint along the $i$--component (Figure \ref{fig:basepoint-pass-ii}), and $G'$ is obtained from $G$ by applying move $(a)$ (or equivalently $G$ is obtained from $G'$ via the move $(b)$). The type of diagrams which may contribute to \eqref{eq:AG-diff-alpha}, are ``mirror reflections'' of diagrams described in Case $(\iota)$.

\begin{figure}[ht]
	\includegraphics[width=.8\textwidth]{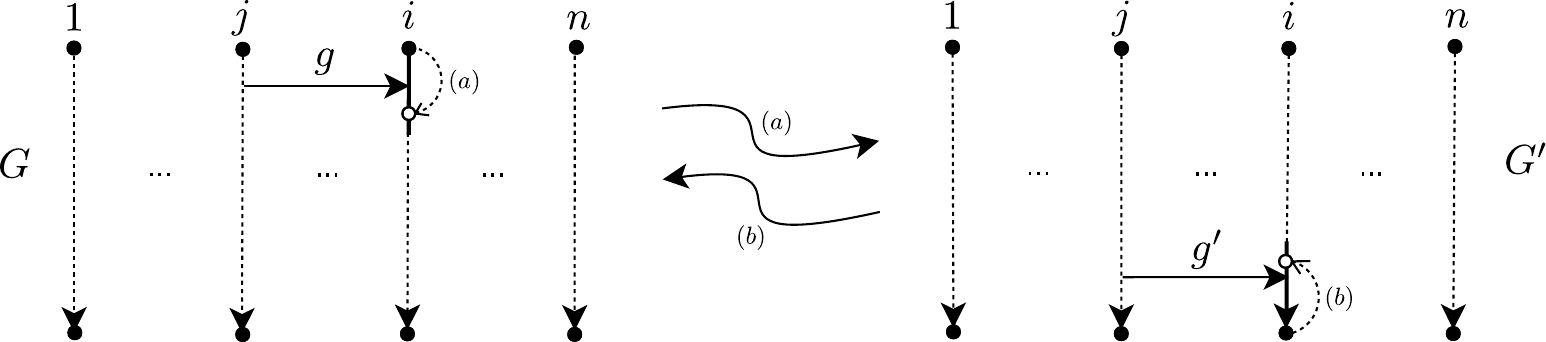}
	\caption{{\bf Case $(\iota\iota)$:} Basepoint moving pass the head of  $g\sim (i,j)$, $i>j$.}\label{fig:basepoint-pass-ii}
\end{figure}
\begin{figure}[ht]
	\includegraphics[width=.35\textwidth]{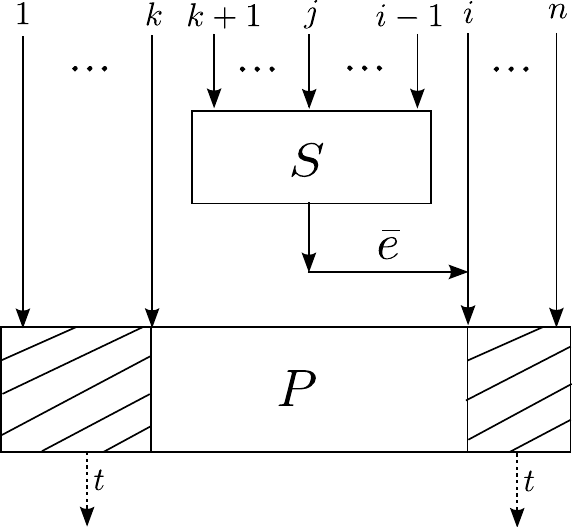}\qquad\qquad \includegraphics[width=.35\textwidth]{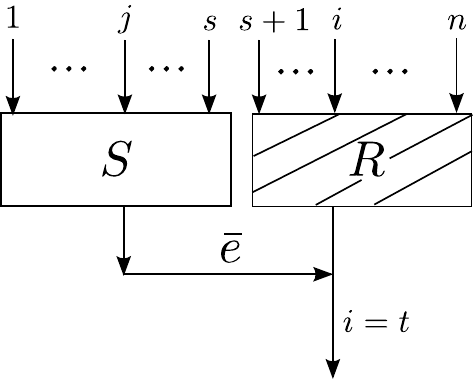}
	\caption{Diagrams in  $(\iota\iota.b)$ and $(\iota\iota.c)$, where $\bar{e}$ corresponds to $\alpha$.}\label{fig:ii_b-ii_c}
\end{figure}%
\no {\bf Case  $(\iota\iota.a)$:} Diagrams $A\in\mathcal{A}(n;t)$ which may yield nonzero both $\langle A, G\rangle_{\alpha\mapsto g}$ and $\langle A, G'\rangle_{\alpha\mapsto g'}$ in \eqref{eq:AG-diff-alpha} need to have an arrow $\alpha\sim (i,j)$ with head on $i$--component  and no other arrow adjacent to that component. It implies that $i$--component is the trunk of $A$ and  $\langle A, G\rangle-\langle A, G'\rangle=\langle A, G\rangle_{\alpha\mapsto g}-\langle A, G'\rangle_{\alpha\mapsto g'}=0$. In fact $A$ admits the decomposition in Figure \ref{fig:ii_b-ii_c}(right) with $R=\varnothing$.

\no {\bf Case  $(\iota\iota.b)$:}  Diagrams  which may yield 
$\langle A, G\rangle_{\alpha\mapsto g}\neq  0$ in \eqref{eq:AG-diff-alpha} need to have an arrow $\alpha\sim (i,j)$ with a top arrowhead along the $i$--component and admit the following decomposition with respect to $\alpha$ (Figure \ref{fig:ii_b-ii_c}(left) and Equation \eqref{eq:A-decomp-bar-e} with $R=\varnothing$):
\begin{equation}\label{eq:A-decomp-ii_b}
A=P\prec_i Q,\qquad Q=\overline{e}\prec_1 S,\qquad P\neq\varnothing,
\end{equation}
where $\alpha$ is the arrow in $\overline{e}$, $v=i(i,A;P)$ and $r(S)=i-j$. Given $I\subset [n]$, $J=([n]-I)\cup \{i\}$, define $\mathcal{A}_{(\iota\iota.b)}(i,j;I;t)$ to be the set of trees in $\mathcal{A}(n;t)$ decomposable according to \eqref{eq:A-decomp-ii_b} with $I(P;A)=I$, and $I(Q;A)=J$. As usual, we let $\mathcal{I}_{(\iota\iota.b)}(i,j)=\{I\subset [n]\ |\ \mathcal{A}_{(\iota\iota.b)}(i,j;I;t)\neq \varnothing\}$.
From \eqref{eq:I(P;Z)-I(Q;Z)}, it follows that $I(P;A)=[1,k ]\cup [i,n]$, for $k<j$, giving
\begin{equation}\label{eq:I_(ii.b)} 
\mathcal{I}_{(\iota\iota.b)}(i,j)=\{[1,k]\cup [i,n]\ |\ k<j\}.
\end{equation}
Since $P$ is the free factor in \eqref{eq:A-decomp-ii_b}, for any $A\in \mathcal{A}_{(\iota\iota.b)}(i,j;I;t)$:
\[
\langle A,G\rangle_{\alpha\mapsto g}=\langle P, G(I)\rangle\, \langle Q, G(J)\rangle_{\alpha\mapsto g},
\]
and since $P\neq \varnothing$; we also have $\langle A,G'\rangle_{\alpha\mapsto g'}=0$ yielding
\begin{equation}\label{eq:z(G)-z(G')-case-ii_b}
 \begin{split}
  & \sum_{A\in \mathcal{A}_{(\iota\iota.b)}(i,j;I;t)}  \sign(A)\langle A,G\rangle_{\alpha\mapsto g}= d_I\, Z_{I;t}(G).
 \end{split}
\end{equation}

\no {\bf Case  $(\iota\iota.c)$:} Diagrams  which may yield 
$\langle A, G'\rangle_{\alpha\mapsto g'}\neq  0$ in \eqref{eq:AG-diff-alpha}, must have no other head/tail below the head of $\alpha$ along the $i$--component of $A$. Such diagrams 
admit the following decomposition with respect to $\alpha$ (Figure \ref{fig:ii_b-ii_c}(right) and Equation \eqref{eq:A-decomp-e} with $P=\varnothing$, $R\neq \varnothing$)
\begin{equation}\label{eq:A-decomp-ii_c}
	A=U\prec_{s+1} R,\qquad  U=\bar{e}\prec_1 S,\qquad S\in \mathcal{A}(s;r)
\end{equation}
where $\alpha$ is an arrow corresponding to $\bar{e}$. Let $\mathcal{A}_{(\iota\iota.c)}(t,j;I)$ denote the set of trees in $\mathcal{A}(n;t)$ decomposable according to \eqref{eq:A-decomp-ii_c} with $I(R;A)=I$, and $I(U;A)=J$, and $\mathcal{I}_{(\iota\iota.c)}(t,j)=\{I\subset [n]\ |\ \mathcal{A}_{(\iota\iota.c)}(t,j;I)\neq \varnothing\}$. From \eqref{eq:I(P;Z)-I(Q;Z)} it follows that $I(R;A)=[s+1,n]$, for $1<s+1\leq t\leq n$ (Figure \ref{fig:ii_b-ii_c}(right)) and  
\begin{equation}\label{eq:I_(ii_c)} 
\mathcal{I}_{(\iota\iota.c)}(i,j)=\{[s+1,n]\ |\ 1<s+1\leq t\leq n\}.
\end{equation}
Since $R$ is the free factor in \eqref{eq:A-decomp-ii_c}, for any $A\in \mathcal{A}_{(\iota\iota.c)}(t,j;I)$ we have $\langle A,G'\rangle_{\alpha\mapsto g'}=\langle R, G(I)\rangle$ $\, \langle U, G(J)\rangle_{\alpha\mapsto g}$, and 
since $R\neq \varnothing$, we also have $\langle A,G\rangle_{\alpha\mapsto g}=0$. Using \eqref{eq:AG-diff-alpha} yields
\begin{equation}\label{eq:z(G)-z(G')-case-ii_c}
\begin{split}
& \sum_{A\in \mathcal{A}_{(\iota\iota.c)}(t,j;I)}  \sign(A)\bigl(\langle A, G\rangle - \langle A, G'\rangle\bigr)= -f_I\, Z_{I;t}(G).
\end{split}
\end{equation}
 Based on $(\iota\iota.a)$, \eqref{eq:z(G)-z(G')-case-ii_b} and \eqref{eq:z(G)-z(G')-case-ii_c}, we obtain
 \begin{equation}\label{eq:z(G)-z(G')-case-ii}
 \begin{split}
  z(G)-z(G') & =\sum_{A\in \mathcal{A}(n;t)} \sign(A)\bigl(\langle A, G\rangle - \langle A, G'\rangle\bigr)\\
 &\ =\sum_{i,j}\sum_{I\in \mathcal{I}_{(\iota\iota.b)}(i,j)} d_I\, Z_{I;t}(G)-\sum_j\sum_{I\in \mathcal{I}_{(\iota\iota.c)}(t,j)} f_I\, Z_{I;t}(G).
 \end{split} 
 \end{equation}
 Therefore, Lemma \ref{lem:basepoint-change} is proven in Case $(\iota\iota)$, where the coefficients $a_{I;j}$ ought to be chosen as $d_I$ or $f_I$ above.
 
\no {\bf Case $(\iota\iota\iota)$:}  Suppose $g=(i,j)$ in $G$, $i<j$ has its arrowtail closest to the basepoint along the $j$th string as pictured in Figure \ref{fig:basepoint-pass-iii}, and $G'$ is obtained from $G$ by applying move $(a)$ (or equivalently $G$ is obtained from $G'$ via the move $(b)$). As in the previous cases,  we need to analyze the right hand side of  \eqref{eq:AG-diff-alpha}.
\begin{figure}[ht]
	\includegraphics[width=.8\textwidth]{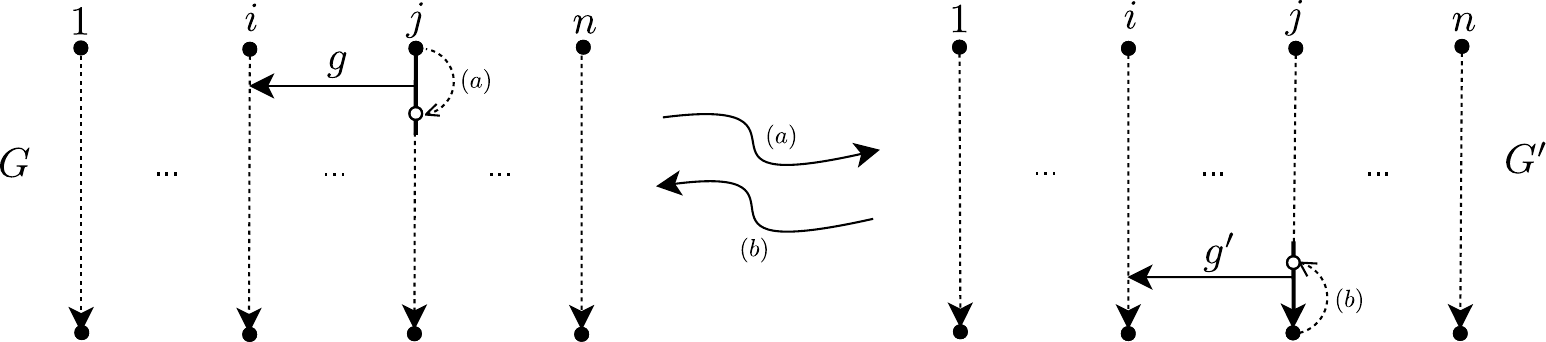}	
	\caption{Case $(\iota\iota\iota)$: Basepoint moving pass the arrow tail of a top $(i,j)$ arrow for $i<j$.}\label{fig:basepoint-pass-iii}
\end{figure}
%
\begin{figure}[ht]
	\includegraphics[width=.35\textwidth]{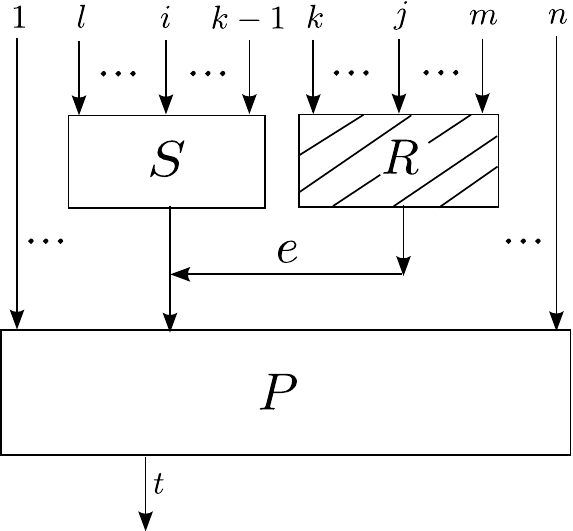}
	\caption{Diagrams in  $(\iota\iota\iota.a)$ and $(\iota\iota\iota.b)$, where $e$ corresponds to $\alpha$.}\label{fig:iii_a-b}
\end{figure}%

\no {\bf Case  $(\iota\iota\iota.a)$:} Because there is no arrowhead/tail above the tail of $g$ along the $j$-component of $G$, any tree $A\in \mathcal{A}(n;t)$  for which $\langle A, G\rangle_{\alpha\mapsto g}\neq 0$ admits the following decomposition (see Figure \ref{fig:iii_a-b} with  $R=\varnothing$ and Lemma \ref{lem:P,Q-decomp})
\[
 A=P\prec_v Q,\qquad Q=e\prec_1 S,\qquad v=i(i,A;P).
\]
Thus there is no other than $\alpha$'s arrowhead/tail in $A$ along the $j$--component of $A$, and 
 terms in $\langle A, G\rangle_{\alpha\mapsto g}$ match those in $\langle A, G'\rangle_{\alpha\mapsto g'}$ yielding
\begin{equation}\label{eq:A,G-diff-zero}
 \langle A, G\rangle_{\alpha\mapsto g}-\langle A, G'\rangle_{\alpha\mapsto g'}=0.
\end{equation}
\no {\bf Case  $(\iota\iota\iota.b)$:} If $\langle A, G'\rangle_{\alpha\mapsto g'}\neq 0$, $A$ needs to have an arrow $\alpha\sim (i,j)$,  with the arrow tail on the $j$--component. If $\alpha$ is the only arrow connected the $j$--component we obtain \eqref{eq:A,G-diff-zero}. Otherwise, $A$ has to decompose as follows   
(Figure \ref{fig:iii_a-b} and Lemma \ref{lem:P,Q-decomp})
\begin{equation}\label{eq:A-decomp-iii_b}
A=U\prec_{w} R,\qquad U=P\prec_{v} (e\prec_1 S),\quad R\neq\varnothing,\qquad v=i(i,U;P),\quad w=i(j,A;U),
\end{equation}
where $\alpha$ corresponds to $e$ and $r(S)+l(R)=j-i-1$. Given $I\subset [n]$, such that $j\in I$, $J=[n]-I$, define $\mathcal{A}_{(\iota\iota\iota.b)}=\mathcal{A}_{(\iota\iota\iota.b)}(i,j;I;t)$ to be the set of trees in $\mathcal{A}(n;t)$ which can be decomposed according to \eqref{eq:A-decomp-iii_b} with $I(R;A)=I$, $I(U;A)=J$ and let
$\mathcal{I}_{(\iota\iota\iota.b)}(i,j)=\{I\subset [n]\ |\ \mathcal{A}_{(\iota\iota\iota.b)}(i,j;I;t)\neq \varnothing\}$.
From \eqref{eq:I(P;Z)-I(Q;Z)}, it follows that $I(R;A)=[k,\ldots,m]$, for $j\leq m\leq n$; $k>i$ yielding 
\begin{equation}\label{eq:I_(iii.b)} 
\mathcal{I}_{(\iota\iota\iota.b)}(i,j)=\{[k,\ldots,m]\ |\ k\leq j\leq m\leq n; k<m; i<k\}.
\end{equation}
The trunk of $A\in \mathcal{A}_{(\iota\iota\iota.b)}$, has index: $t<l$, $t>m$ or $t=i$.
Further, for any $A\in \mathcal{A}_{(\iota\iota\iota.b)}$ we have
$\langle A,G'\rangle_{\alpha\mapsto g'}=\langle R, G(I)\rangle\, \langle U, G(J)\rangle_{\alpha\mapsto g'}$, and 
since $R\neq\varnothing$, we have $\langle A,G\rangle_{\alpha\mapsto g}=0$. Therefore $R$ is a free factor in \eqref{eq:A-decomp-iii_b} and
\begin{equation}\label{eq:z(G)-z(G')-case-iii_b}
\sum_{A\in \mathcal{A}_{(\iota\iota\iota.b)}(i,j;I;t)}  \sign(A)\bigl(\langle A, G\rangle - \langle A, G'\rangle\bigr)=- h_I\, Z_{I;j}(G).
\end{equation}
 Based on $(\iota\iota\iota.a)$ and  $(\iota\iota\iota.b)$ we obtain
 \begin{equation}\label{eq:z(G)-z(G')-case-iii}
 \begin{split}
 z(G)-z(G') & =\sum_{A\in \mathcal{A}(n;t)} \sign(A)\bigl(\langle A, G\rangle - \langle A, G'\rangle\bigr) =-\sum_{i,j}\sum_{I\in \mathcal{I}_{(\iota\iota\iota.b)}(i,j)} h_I\, Z_{I;j}(G),
 \end{split} 
 \end{equation}
 which ends the proof of Lemma \ref{lem:basepoint-change} in Case $(\iota\iota\iota)$.

\no {\bf Case $(\iota v)$:}  Suppose $g=(i,j)$ in $G$, $i>j$ has its arrowtail closest to the basepoint along the $j$th string as pictured in Figure \ref{fig:basepoint-pass-iv}, and $G'$ is obtained from $G$ applying move $(a)$ (or equivalently $G$ is obtained from $G'$ via the move $(b)$).  
\begin{figure}[ht]
	\includegraphics[width=.8\textwidth]{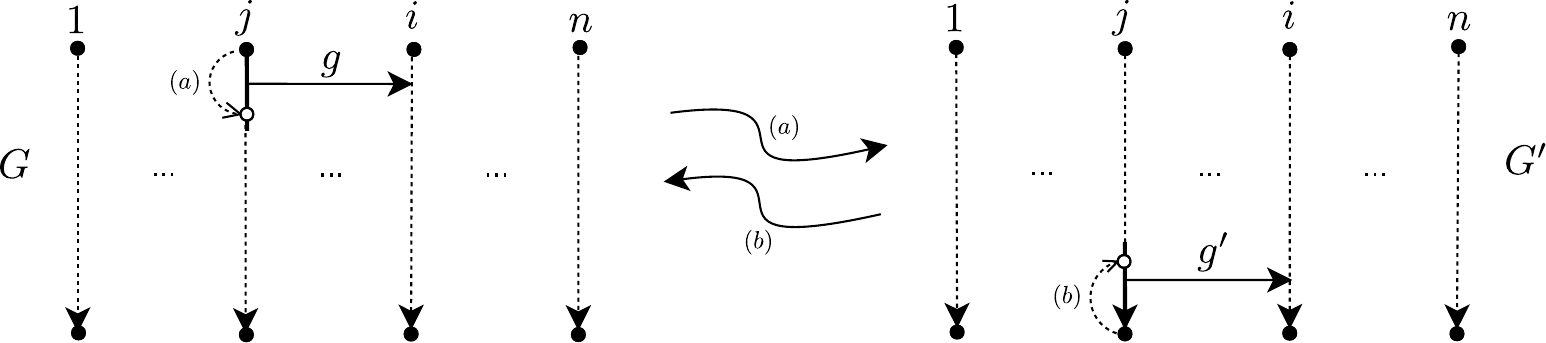}
	\caption{Case $(\iota v)$: Basepoint moving pass the arrowtail of $g=(i,j)$ for $i>j$.}\label{fig:basepoint-pass-iv}
\end{figure}
%
\begin{figure}[ht]
	\includegraphics[width=.35\textwidth]{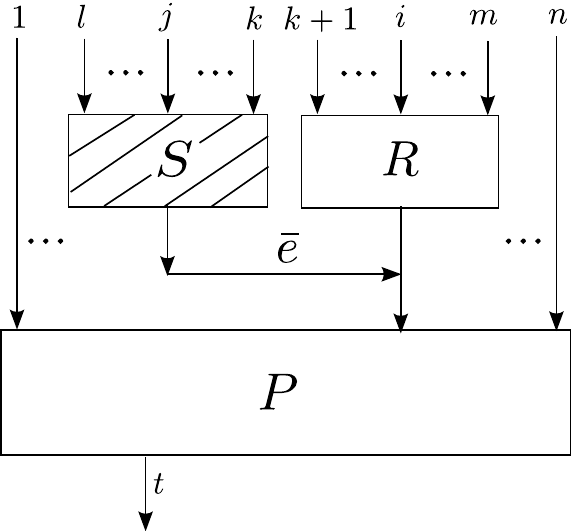}
	\caption{Diagrams in  $(\iota v.a)$ and $(\iota v.b)$, where $\bar{e}$ corresponds to $\alpha$.}\label{fig:iv_a-b}
\end{figure}

\no {\bf Case $(\iota v.a)$:} Any tree $A$ giving $\langle A, G\rangle_{\alpha\mapsto g}\neq 0$, admits the following decomposition, Figure \ref{fig:iv_a-b} with $S=\varnothing$;
\[
A=P\prec_v Q,\qquad Q=\bar{e}\prec_2 R,\qquad v=i(i,A;P).
\]
As in Case $(\iota\iota\iota.a)$, terms in the sum $\langle A, G\rangle_{\alpha\mapsto g}$ match those in $\langle A, G'\rangle_{\alpha\mapsto g'}$ yielding \eqref{eq:A,G-diff-zero}.

\no {\bf Case $(\iota v.b)$:} If $\langle A, G'\rangle_{\alpha\mapsto g'}\neq 0$, $A$ needs to have an arrow $\alpha\sim (i,j)$,  with the bottom arrow tail on the $j$--component. Any relevant tree diagram $A$ in this case not included in Case $(\iota v.a)$ decomposes as pictured in Figure \ref{fig:iv_a-b}, i.e.
\begin{equation}\label{eq:A-decomp-iv_b}
A=U\prec_{w} S,\qquad U=P\prec_{v} (\bar{e}\prec_2 R),\quad S\neq\varnothing,\qquad v=i(i-|S|,U;P),\quad w=i(j,A;U),
\end{equation}
where $\alpha$ corresponds to $\bar{e}$ and $r(S)+l(R)=j-i-1$. Given $I\subset [n]$, such that $j\in I$, $J=[n]-I$,   define $\mathcal{A}_{(\iota v.b)}=\mathcal{A}_{(\iota v.b)}(i,j;I;t)$ to be the set of trees which can be decomposed according to \eqref{eq:A-decomp-iv_b} with $I(S;A)=I$, $I(U;A)=J$ and 
$\mathcal{I}_{(\iota v.b)}(i,j)=\{I\subset [n]\ |\ \mathcal{A}_{(\iota v.b)}(i,j;I;t)\neq \varnothing\}$.
From \eqref{eq:I(P;Z)-I(Q;Z)}, it follows that $I(S;A)=[l,\ldots,k]$, for $j\leq k< i$; $1\leq l\leq j$, giving
\begin{equation}\label{eq:I_(iv.b)} 
\mathcal{I}_{(\iota v.b)}(i,j)=\{[l,\ldots,k]\ |\ 1\leq l\leq j\leq k< i, l<k\}.
\end{equation}
The trunk of $A$, has index $t<l$, $t>m$ or $t=i$.
For any $A\in \mathcal{A}_{(\iota v.b)}$ we obtain
$\langle A,G'\rangle_{\alpha\mapsto g'}=\langle S, G(I)\rangle\, \langle U, G(J)\rangle_{\alpha\mapsto g'}$,
and since $S\neq\varnothing$, we have $\langle A,G\rangle_{\alpha\mapsto g}=0$. Therefore, $S$ is a free factor in \eqref{eq:A-decomp-iv_b} yielding
\begin{equation}\label{eq:z(G)-z(G')-case-iv_b}
\sum_{A\in \mathcal{A}_{(\iota v.b)}(i,j;I;t)}  \sign(A)\bigl(\langle A, G\rangle - \langle A, G'\rangle\bigr)=- w_I\, Z_{I;j}(G).
\end{equation}
Based on $(\iota v.a)$ and  $(\iota v.b)$, we obtain
 \begin{equation}\label{eq:z(G)-z(G')-case-iv}
 \begin{split}
 z(G)-z(G') & =\sum_{A\in \mathcal{A}(n;t)} \sign(A)\bigl(\langle A, G\rangle - \langle A, G'\rangle\bigr) =-\sum_{i,j}\sum_{I\in \mathcal{I}_{(\iota v.b)}(i,j)} w_I\, Z_{I;j}(G),
 \end{split} 
 \end{equation}
 which ends the proof of  Lemma \ref{lem:basepoint-change} in Case $(\iota v)$.

Recapping Cases $(\iota)$--$(\iota v)$ we obtain the identity \eqref{eq:z(G)-z(G')} of Lemma \ref{lem:basepoint-change}, where the index sets $\mathcal{I}_k$ are obtained from: \eqref{eq:I_(i.b)}, \eqref{eq:I_(i.c)}, \eqref{eq:I_(ii.b)}, \eqref{eq:I_(iii.b)}, \eqref{eq:I_(iv.b)}. Specifically,
\[
\mathcal{I}_t  =\bigcup_{i,j}\mathcal{I}_{(\iota.b)}(i,j)\cup \mathcal{I}_{(\iota.c)}(t,j)\cup \mathcal{I}_{(\iota\iota.b)}(i,j)\cup \mathcal{I}_{(\iota\iota.c)}(t,j),\qquad
\mathcal{I}_j=\bigcup_{i}\mathcal{I}_{(\iota\iota\iota.b)}(i,j)\cup \mathcal{I}_{(\iota v.b)}(i,j). \qedhere
\]
\end{proof}

\subsection{Reidemeister moves}
 Figure \ref{fig:reidemeister-moves} shows the Reidemeister moves: $\mathbf{r1}$ through $\mathbf{r3}$ of link diagrams away from the basepoints and the corresponding local arrow changes in their Gauss diagrams. A goal for this subsection is to prove the following 
\begin{figure}[!ht] 
	\centering
	\includegraphics[width=0.8\textwidth]{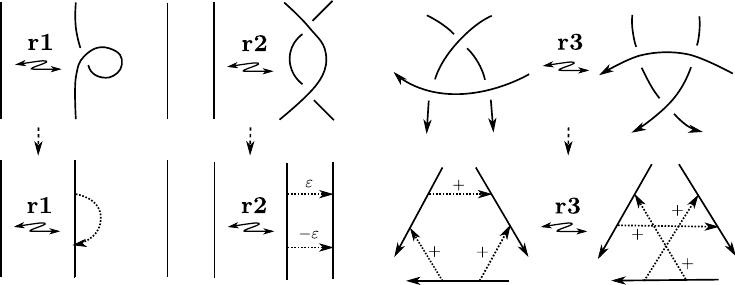}
	\caption{Reidemeister moves $\mathbf{r1}$--$\mathbf{r3}$ locally on the link diagram (far from the basepoints) and the corresponding local changes of Gauss diagrams (c.f. \cite{Kravchenko-Polyak:2011}).}\label{fig:reidemeister-moves}
\end{figure}	
\begin{lem}\label{lem:reidemeister}
 Given a Gauss diagram $G$ of a closed based $n$--component link, let $G'$ be a diagram obtained from $G$ by applying locally one of the moves  $\mathbf{r1}$--$\mathbf{r3}$ away from the basepoints of the link components. Then for any $t$ we have
 \[
  \langle Z_{n;t}, G\rangle=\langle Z_{n;t}, G'\rangle.
 \]
\end{lem}
The proof given below follows closely the argument of Kravchenko and Polyak \cite{Kravchenko-Polyak:2011} and is included mostly for completeness.

\begin{proof}[Proof of Lemma \ref{lem:reidemeister}]

\no {\bf Case $\mathbf{r1}$:} Let $G$ be a Gauss diagram of a based link $L$, and $G'$ a diagram obtained by the move $\mathbf{r1}$, i.e. $G$ and $G'$ differ by a single arrow $g$ shown in Figure \ref{fig:reidemeister-moves}.
Since none of the tree diagrams $A\in \mathcal{A}(n;1)$ have arrows with the heads and tails on the same segment, and there exists no embedding $\phi:A\longmapsto G$, such that $\phi(\alpha)=g$ for some $\alpha\in A$. Therefore every such embedding is also an embedding into $G'$ and we obtain $\langle A, G\rangle = \langle A, G'\rangle$, proving $z(G)=z(G')$  for $\mathbf{r1}$. 

\medskip
\no {\bf Case $\mathbf{r2}$:} The diagram $G'$ is obtained from $G$ by adding locally two parallel arrows $g_{+}$ and $g_{-}$ of opposite sign i.e. both arrows have their head on the $i$th string and the tail on the $j$th string of $G'$. By uniqueness of arrows in tree diagrams each embedding  $\phi:A\longmapsto G'$ can map $\alpha\sim (i,j)$ in $A$ to either $g_{+}$ or $g_{-}$, therefore 
\begin{equation}\label{eq:r2-AG'}
\langle A, G'\rangle=\langle A, G'\rangle_{\alpha\to g_{+}}+\langle A, G'\rangle_{\alpha\to g_{-}}+ \langle A, G'\rangle_{\alpha\not\to \{g_{+},g_{-}\}}.
\end{equation}
\no Every $\phi':A\longmapsto G'$, $\phi'(\alpha)\not\in \{g_{+},g_{-}\}$, factors through the inclusion $\iota:G'-\{g_{+},g_{-}\}\longmapsto G'$, i.e. $\phi'=\iota\circ\phi$. Since $G=G'-\{g_{+},g_{-}\}$: $\phi:A\longmapsto G$ is an embedding into $G$ thus $\sign(\phi')=\sign(\phi)$ and we obtain  $\langle A, G\rangle=\langle A, G'\rangle_{\alpha\not\to \{g_{+},g_{-}\}}$. Next consider $\phi':A\longmapsto G'$, $\phi'(\alpha)=g_{+}$, i.e. an embedding which may contribute to $\langle A, G'\rangle_{\alpha\to g_{+}}$, every such embedding can be redefined as $\phi:A\longmapsto G'$, where $\phi|_{A-\{\alpha\}}=\phi'|_{A-\{\alpha\}}$ and $\phi(\alpha)=g_{-}$, since $g_{+}$ and $g_{-}$ have opposite signs in $G'$ we have $\sign(\phi)=-\sign(\phi')$. Therefore terms in $\langle A, G'\rangle_{\alpha\to g_{+}}$ are in one--to--one correspondence with terms in $\langle A, G'\rangle_{\alpha\to g_{-}}$, but with opposite sign, yielding $\langle A, G'\rangle_{\alpha\to g_{+}}+\langle A, G'\rangle_{\alpha\to g_{-}}=0$. Collecting the above facts we obtain: $\langle A, G'\rangle=\langle A, G\rangle$, proving $z(G)=z(G')$ for $\mathbf{r2}$.

\medskip
\no {\bf Case $\mathbf{r3}$:} This is the most involved case, which includes subcases corresponding to the order of components $i_1$, $i_2$ and $i_3$ of $G$ on which the {\bf r3}--move is performed. 

 The rearrangement of arrows $g\sim (i_1,i_2)$, $h\sim (i_1, i_3)$ and $k 
 \sim (i_2, i_3)$ under the $\mathbf{r3}$ move is shown in Figure 
\ref{fig:reidemeister-moves}. Given any tree diagram $A\in \mathcal{A}(n;t)$, let 
\begin{equation}\label{eq:alpha-beta-gamma} 
\alpha\sim (i_1, i_2),\quad \beta\sim (i_1, i_3),\quad \gamma\sim (i_2,i_3).
\end{equation}
 Note that all three arrows cannot be in $A$, as it would contradict planarity of $A$. Let us determine those trees $A$, which can contribute to $\langle A, G\rangle-\langle A, G'\rangle$. Clearly, if $A$ contains none of the arrows in \eqref{eq:alpha-beta-gamma} any embedding $\phi:A\longmapsto G$ factors through the inclusion $G-\{g,h,k\}\hookrightarrow G$ 
 and since $G-\{g,h,k\}=G'-\{g',h',k'\}$, $\phi$ also embeds $A$ in $G'$ then, analogously as in the case of {\bf r2}, we conclude $\langle A, G\rangle=\langle A, G'\rangle$. 
 
 If $A$ contains exactly one of the arrows in \eqref{eq:alpha-beta-gamma}, without loss of generality, suppose $\beta\in A$, then we have an analog of \eqref{eq:AG-diff-alpha}:
 \[
\langle A, G\rangle - \langle A, G'\rangle  = \bigl(\langle A, G\rangle_{\beta\mapsto h}-\langle A, G'\rangle_{\beta\mapsto h'}\bigr)+\bigl(\langle A, G\rangle_{\beta\not\mapsto h}-\langle A, G'\rangle_{\beta\not\mapsto h'}\bigr).
 \]
 The second term of the above sum vanishes by the same reasoning as in the previous paragraph. The first term also vanishes because every embedding $\phi:A\longmapsto G$, $\phi(\beta)=h$ can be locally redefined as  $\phi':A\longmapsto G'$, $\phi'|_{A-\{\beta\}}=\phi|_{A-\{\beta\}}$, $\phi'(\beta)=h'$ and $\sign(\phi)=\sign(\phi')$.
\begin{figure}[!ht] 
	\begin{overpic}[width=0.8\textwidth]{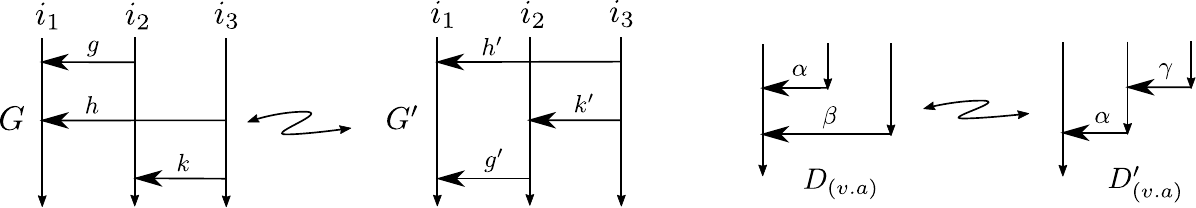}
\end{overpic}
	\caption{Case $(v.a)$: Reidemeister $\mathbf{r3}$ for $i_1<i_2<i_3$. Matching arrows: $\{\alpha\to g, \beta\to h\}$, $\{\alpha\to g', \gamma\to k'\}$.}\label{fig:r3-v_a}
\end{figure}	
 \smallskip
The remaining case is when $A$ contains exactly two of the arrows from \eqref{eq:alpha-beta-gamma}. The following cases depend on the ordering of strings in $G$. 

\no {\bf Case $(v.a)$: $i_1<i_2<i_3$} and either $(a):$ $\{\alpha, \beta\}\subset A$, $(b):$ $\{\alpha, \gamma\}\subset A$ or $(c)$ $\{\beta,\gamma\}\subset A$. 

First suppose $A$ satisfies $(a)$, given an embedding $\phi:A\longmapsto G$, the only possibilities are $1^\circ$: $ \phi(A)\subset G-\{g,h,k\}$, $2^\circ$:  $\phi(\alpha)=g$ and $\phi(\beta)\neq h$ or  $\phi(\alpha)\neq g$ and $\phi(\beta)= h$;  $3^\circ$: $\phi(\alpha)=g$ and $\phi(\alpha)=h$. Therefore, we have
\begin{equation}\label{eq:r3-1-AG-alpha-beta}
  \langle A,G\rangle = \langle A,G\rangle_{\{\alpha,\beta\}\not\to \{g,h,k\}}+\langle A,G\rangle_{\alpha\to g, \beta\not\to h}+\langle A,G\rangle_{\alpha\not\to g, \beta\to h}+\langle A,G\rangle_{\alpha\to g, \beta\to h}.
\end{equation}
On the other hand there is no embedding $\phi':A\longmapsto G'$, with $\phi'(\alpha)=g'$ and $\phi'(\beta)=h'$, and  \eqref{eq:r3-1-AG-alpha-beta} becomes
\begin{equation}\label{eq:r3-1-AG'-alpha-beta}
\langle A,G'\rangle = \langle A,G'\rangle_{\{\alpha,\beta\}\not\to \{g',h',k'\}}+\langle A,G'\rangle_{\alpha\to g', \beta\not\to h'}+\langle A,G'\rangle_{\alpha\not\to g', \beta\to h'}.
\end{equation}
 \begin{figure}[!ht] 
 	\begin{overpic}[width=0.8\textwidth]{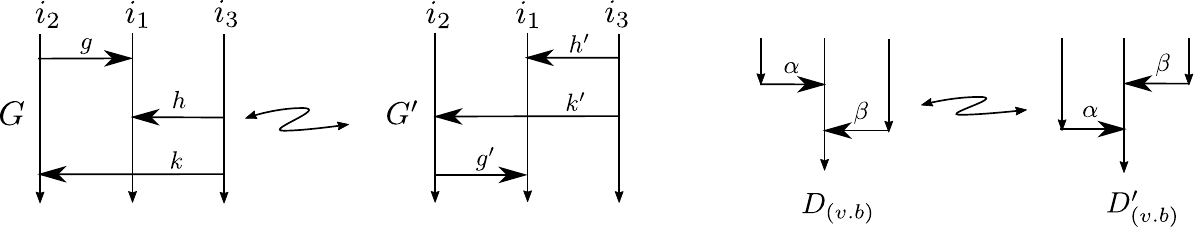}
 	\end{overpic}
 	\caption{Case $(v.b)$: Reidemeister $\mathbf{r3}$ for $i_2<i_1<i_3$.}\label{fig:r3-v_b}
 \end{figure}	
 \begin{figure}[!ht] 
 	\begin{overpic}[width=0.8\textwidth]{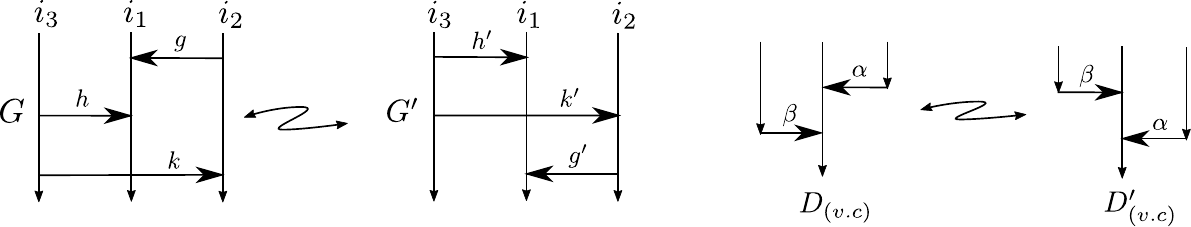}
 	\end{overpic}
 	\caption{Case $(v.c)$: Reidemeister $\mathbf{r3}$ for $i_3<i_1<i_2$.}\label{fig:r3-v_c}
 \end{figure}	
 \begin{figure}[!ht] 
 	\begin{overpic}[width=0.8\textwidth]{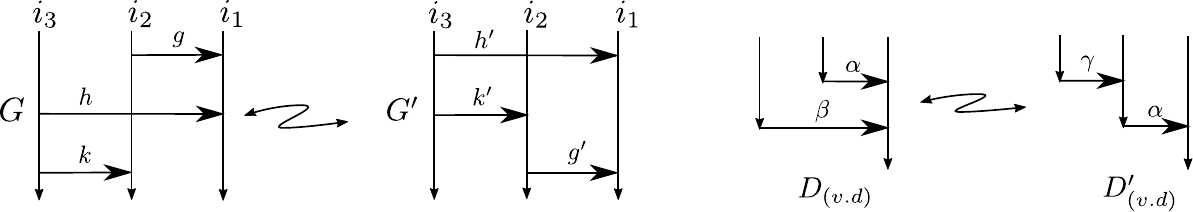}
 	\end{overpic}
 	\caption{Case $(v.d)$: Reidemeister $\mathbf{r3}$ for $i_3<i_2<i_1$.}\label{fig:r3-v_d}
 \end{figure}	

\no It is clear that the first three terms of \eqref{eq:r3-1-AG-alpha-beta} and \eqref{eq:r3-1-AG'-alpha-beta} agree, giving 
\begin{equation}\label{eq:r3-1-AG-AG'-alpha-beta}
 \langle A,G\rangle-\langle A,G'\rangle=\langle A,G\rangle_{\alpha\to g, \beta\to h}.
\end{equation}
In order for $A$ to contribute to $\langle A,G\rangle_{\alpha\to g, \beta\to h}$, $A$ must contain a subtree $D_{(v.a)}$ shown in Figure \ref{fig:r3-v_a}, (because the move is local, and there are no other arrows neighboring $g$, $h$, $k$). Denote the set of such diagrams by 
\[
 \mathcal{A}_{D_{(v.a)}}=\bigl\{ A\in \mathcal{A}(n;t)\ |\ D_{(v.a)}\subset A, I(D_{(v.a)};A)=\{i_1,i_2,i_3\}\bigr\}.
\]
\no In Case $(b)$: $\{\alpha,\gamma\}\subset A$, fully analogous steps as in \eqref{eq:r3-1-AG-alpha-beta}--\eqref{eq:r3-1-AG-AG'-alpha-beta}  yield
\begin{equation}\label{eq:r3-1-AG-AG'-alpha-gamma}
\langle A,G\rangle-\langle A,G'\rangle=-\langle A,G'\rangle_{\alpha\to g', \gamma\to k'}.
\end{equation}
For $A$ to contribute to $\langle A,G\rangle_{\alpha\to g', \gamma\to k'}$, it must contain a subtree $D'_{(v.a)}$ shown in Figure \ref{fig:r3-v_a}, denote the set of such diagrams by 
\[
\mathcal{A}_{D'_{(v.a)}}=\bigl\{ A\in \mathcal{A}(n;t)\ |\ D'_{(v.a)}\subset A, I(D'_{(v.a)};A)=\{i_1,i_2,i_3\}\bigr\}.
\]
\no In Case $(c):$ $\{\beta,\gamma\}\subset A$, there is no embedding $\phi:A\longrightarrow G$, with $\phi(\beta)=h$ and $\phi(\gamma)=k$, as well as no embedding $\phi':A\longrightarrow G'$, with $\phi'(\beta)=h'$ and $\phi'(\gamma)=k'$, yielding
\begin{equation}\label{eq:r3-1-AG-AG'-beta-gamma}
\langle A,G\rangle-\langle A,G'\rangle=0.
\end{equation}
As a result we obtain
\begin{equation}\label{eq:z(G)-z(G')-r3-1}
z(G)-z(G')=\sum_{A\in  \mathcal{A}_{D_{(v.a)}}} \langle A,G\rangle_{\alpha\to g, \beta\to h}-\sum_{A\in  \mathcal{A}_{D'_{(v.a)}}} \langle A,G'\rangle_{\alpha\to g', \beta\to k'}.
\end{equation}
Observe that there is a bijection  $f:\mathcal{A}_{D_{(v.a)}}\longrightarrow\mathcal{A}_{D'_{(v.a)}}$, for a given $A\in \mathcal{A}_{D_{(v.a)}}$ defined simply by replacing $D_{(v.a)}$ subdiagram with $D'_{(v.a)}$. For every $A\in \mathcal{A}_{D_{(v.a)}}$, and $f(A)\in \mathcal{A}_{D'_{(v.a)}}$,  an embedding $\phi:A\longmapsto G$ may be redefined as $\phi':f(A)\longmapsto G'$, where $\phi'|_{f(A)-\{\alpha,\gamma\}}=\phi|_{A-\{\alpha,\beta\}}$ (since $f(A)-\{\alpha,\gamma\}=A-\{\alpha,\beta\}$), and $\phi'(\alpha)=g'$ and $\phi'(\gamma)=k'$. Since $\sign(D_{(v.a)})=\sign(D'_{(v.a)})$ we have $\sign(\phi)=\sign(\phi')$ (all arrows $g$, $h$, $k$ in $G$, and $g'$, $h'$, $k'$ in $G'$ have the positive sign). In turn we obtain
\[
\langle A,G\rangle_{\alpha\to g, \beta\to h}=\langle f(A),G'\rangle_{\alpha\to g', \beta\to k'}.
\]
Since $f$ is a bijection we may conclude that the right hand side of \eqref{eq:z(G)-z(G')-r3-1} vanishes and $z(G)=z(G')$ proving the claim in Case $(v.a)$.

\no {\bf Cases $(v.b)$: $i_2<i_1<i_3$, $(v.c)$: $i_3<i_1<i_2$, $(v.d)$: $i_3<i_2<i_1$}.
 For these Cases the argument is the same as in Case $(v.a)$, except $\mathcal{A}_{D_{(v.\ast)}}$ and $\mathcal{A}_{D'_{(v.\ast)}}$ must be used in place of $\mathcal{A}_{D_{(v.a)}}$ and $\mathcal{A}_{D'_{(v.a)}}$ as shown in Figures \ref{fig:r3-v_b}, \ref{fig:r3-v_c} and \ref{fig:r3-v_d}.

\no {\bf Cases $(v\iota.a)$: $i_1<i_3<i_2$, and $(v\iota.b)$: $i_2<i_3<i_1$};
As before the two subcases (Figure \ref{fig:r3-vi_a} and Figure \ref{fig:r3-vi_b}) are analogous, let us consider $(v\iota.a)$ in detail, see Figure \ref{fig:r3-vi_a}. Either $(a):$ $\{\alpha, \beta\}\subset A$, $(b):$ $\{\alpha, \gamma\}\subset A$ or $(c)$ $\{\beta,\gamma\}\subset A$. In Case $(a)$, observe that for any $A\in \mathcal{A}(n;t)$, there is no embedding $\phi:A\longmapsto G$ such that $\phi(\alpha)=h$, $\phi(\beta)=g$. Using an analogous identity as Equation \eqref{eq:r3-1-AG-alpha-beta}, we obtain
\begin{equation}\label{eq:r3-2-AG-AG'-alpha-beta}
\langle A,G\rangle-\langle A,G'\rangle=-\langle A,G'\rangle_{\alpha\to g', \beta\to h'}.
\end{equation}
The set of diagrams which may contribute to the left hand side in \eqref{eq:r3-2-AG-AG'-alpha-beta} is denoted by 
\[
\mathcal{A}_{D'_{(v\iota.a)}}=\bigl\{ A\in \mathcal{A}(n;t)\ |\ D'_{(v\iota.a)}\subset A, I(D'_{(v\iota.a)};A)=\{i_1,i_3,i_2\}\bigr\}.
\]
In Case $(b)$ there is no embedding which may contribute to $\langle A,G\rangle_{\alpha\to g, \gamma\to k}$ thus
\begin{equation}\label{eq:r3-2-AG-AG'-alpha-gamma}
\langle A,G\rangle-\langle A,G'\rangle=-\langle A,G'\rangle_{\alpha\to g', \gamma\to k'}.
\end{equation}
\[
\mathcal{A}_{D''_{(v\iota.a)}}=\bigl\{ A\in \mathcal{A}(n;t)\ |\ D''_{(v\iota.a)}\subset A, I(D''_{(v\iota.a)};A)=\{i_1,i_3,i_2\}\bigr\}.
\]
In Case $(c)$ there is no embedding of $A$ which may contribute either to $\langle A,G\rangle_{\beta\to h, \gamma\to k}$ or to $\langle A,G'\rangle_{\beta\to h', \gamma\to k'}$.  
As a result we obtain
\begin{equation}\label{eq:z(G)-z(G')-r3-2}
z(G)-z(G')=-\sum_{A\in  \mathcal{A}_{D'_{(v\iota.a)}}} \langle A,G'\rangle_{\alpha\to g', \beta\to h'}-\sum_{A\in  \mathcal{A}_{D''_{(v\iota.a)}}} \langle A,G'\rangle_{\alpha\to g', \gamma\to k'}.
\end{equation}
\begin{figure}[!ht] 
	\begin{overpic}[width=0.85\textwidth]{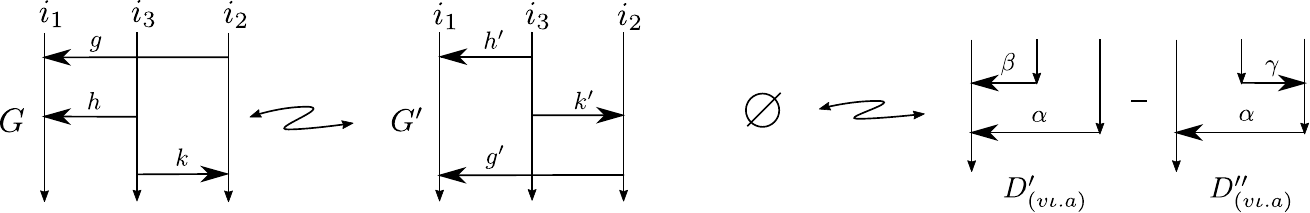}	
	\end{overpic}
	\caption{Case $(v\iota.a)$: Reidemeister $\mathbf{r3}$ for $i_1<i_3<i_2$.}\label{fig:r3-vi_a}
\end{figure}	
\smallskip

\begin{figure}[!ht] 
	\begin{overpic}[width=0.85\textwidth]{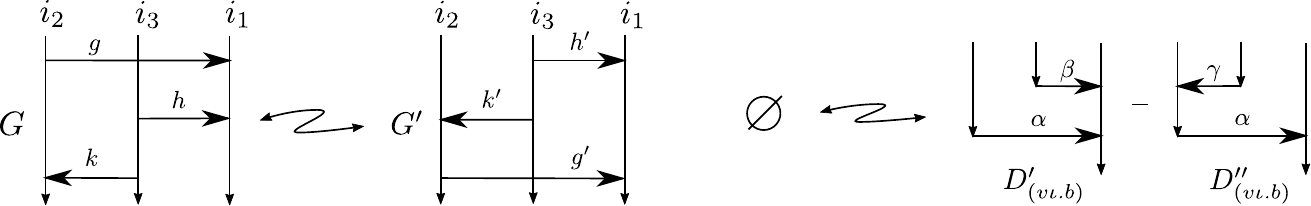}
	\end{overpic}
	\caption{Case $(v\iota.b)$: Reidemeister $\mathbf{r3}$ for $i_2<i_3<i_1$.}\label{fig:r3-vi_b}
\end{figure}
Since diagrams in $\mathcal{A}_{D'_{(v\iota.a)}}$($\mathcal{A}_{D''_{(v\iota.a)}}$) all must contain $D'_{(v\iota.a)}$($D''_{(v\iota.a)}$) as a local subdiagram, we observe that there is a bijection $f:\mathcal{A}_{D'_{(v\iota.a)}}\longmapsto \mathcal{A}_{D''_{(v\iota.a)}}$ given by replacing $D'_{(v\iota.a)}$ in $A\in \mathcal{A}_{D'_{(v\iota.a)}}$ with $D'_{(v\iota.a)}$. Moreover, for any embedding $\phi:A\longmapsto G'$ there exists a corresponding $\phi':f(A)\longmapsto G'$, defined in the obvious way, and since $\sign(D'_{(v\iota.a)})=-\sign(D''_{(v\iota.a)})$, we have $\sign(\phi)=-\sign(\phi')$, and therefore
\[
\langle A,G'\rangle_{\alpha\to g', \beta\to h'}=-\langle f(A),G'\rangle_{\alpha\to g', \gamma\to k'}.
\]
As a result the right hand side of \eqref{eq:z(G)-z(G')-r3-2} vanishes, and $z(G)=z(G')$ as required. 

Cases above justify the claim of Lemma \ref{lem:reidemeister}.
\end{proof}

\begin{proof}[Proof of Main Theorem]
For convenience, and without loss of generality, we will work with $I=[n]$ and $Z_{n;t}$. 
Let $L$ be an $n$--component based link and $G=G_L$ its Gauss diagram, recall  $Z_{n;t}(L)=\langle Z_{n;t}, G_L\rangle$, according to \eqref{eq:Z_I;j} and \eqref{eq:Z_I;j-poly}.
Suppose $G'=G'_L$ is a diagram, obtained by moving basepoints along the components of $L$ in an arbitrary way. Naturally, there is a sequence of diagrams 
\[
G=G_0,\ G_1,\ \ldots,\ G_k=G',
\]
where $G_i$ and $G_{i+1}$, $0\leq i<k$ differ just by a single crossing pass as in Cases $(\iota)$--$(\iota v)$ of Lemma \ref{lem:basepoint-change}. From Lemma \ref{lem:basepoint-change} we obtain for each $i$:
\[
	\langle Z_{n;t}, G_i\rangle-\langle Z_{n;t}, G_{i+1}\rangle=\sum^n_{j=1}\sum_{J\in \mathcal{I}_j} a_{J;j} \langle Z_{J;j}, G_i\rangle.
\]
Therefore, $\langle Z_{n;t},G\rangle-\langle Z_{n;t},G'\rangle=\sum^{k-1}_{i=0} (\langle Z_{n;t},G_i\rangle-\langle Z_{n;t},G_{i+1}\rangle)$ and substituting in the above equation, 
we may express  $\langle Z_{n;t},G\rangle-\langle Z_{n;t},G'\rangle$ as a linear combination of $\{\langle Z_{I;j},G_i\rangle\ | 0\leq i<k\}$, i.e.
\[
\langle Z_{n;t},G\rangle-\langle Z_{n;t},G'\rangle=\sum^{k-2}_{i=0}\sum_{J,j} a_{J;j,i} \langle Z_{J;j},G_i\rangle,\qquad a_{J;j,i}\in \mathbb{Z}, 
\]
where the term $\langle Z_{I;j}, G_k\rangle$ is excluded from the sum.
Applying this step and \eqref{eq:z(G)-z(G')} inductively to terms in the sum above we conclude that $\langle Z_{n;t},G\rangle-\langle Z_{n;t},G'\rangle$ is a linear combination of integers from the set
\begin{equation}\label{eq:Delta-set}
\Gamma_Z=\{\langle Z_{J;k},G\rangle \ |\ J\subsetneq [n], k\in J\},
\end{equation}
(note that the index sets defined in \eqref{eq2:I_j} of Lemma \ref{lem:basepoint-change}, subjected to the above inductive process yield all subsets $J\subsetneq [n]$ with trunks $k\in J$). 
Letting $\Delta_Z(n;t):=\gcd(\Gamma_Z)$, we conclude that 
\begin{equation}\label{eq:bar-Z-cong}
\langle Z_{n;t}, G_L\rangle \equiv \langle Z_{n;t}, G'_L\rangle \mod \Delta_Z(n;t),
\end{equation}
and thus $\bar{Z}_{n;t}(L)$ is invariant under the basepoint changes.

 Invariance of $\overline{Z}_{n;t}(G_L)$ under Reidemeister moves follows immediately from Lemma \ref{lem:reidemeister}, if there are no basepoints on the interacting strands locally. If a move involves basepoints, we may slide them away obtaining a diagram $G'_L$, then using \eqref{eq:bar-Z-cong},  Lemma \ref{lem:reidemeister} applied to $G'_L$ yields the claim. The link homotopy invariance of $\overline{Z}_{n;t}(G_L)$ follows from the fact that tree diagram do not have arrows with a head and tail on the same component. Thus crossing changes within a given component of $L$ do not affect the value of $\overline{Z}_{n;t}(L)$. 
\end{proof}
\begin{proof}[Proof of Corollary \ref{cor:bar-mu-bar-Z}]
	Recall that we are trying to show that
	\begin{equation}\label{eq:bar-Z=bar-mu}
	 \overline{Z}_{n;1}(L)=\overline{\mu}_{n;1}(L).
	\end{equation}
After a possible link homotopy of $L$ we may assume $L\simeq \widehat{\ell}$ for some string link $\ell$, c.f. \cite{Habegger-Lin:1990}. Using \eqref{eq:bar-mu-sigma} and 
Theorem \ref{thm:tree=mu},
\[
\overline{\mu}_{n;1}(L)=\langle Z_{n;1}, G_\ell\rangle \mod \Delta_\mu(n;1),
\]
and from Main Theorem, 
\[
\overline{Z}_{n;1}(L)=\langle Z_{n;1}, G_\ell\rangle \mod \Delta_Z(n;1).
\]
Therefore, it suffices to prove 
\begin{equation}\label{eq:D_mu=D_Z}
 \Delta_Z(n;1)=\Delta_\mu(n;1).
\end{equation}
For that purpose, first observe that by {\bf (s1)} (Equation \eqref{eq:s1}) we may consider a subset $\Gamma'_Z\subset \Gamma_Z$:
\[
 \Gamma'_Z=\{\langle Z_{J;k^+},G\rangle, \langle Z_{J;k^-},G\rangle\ |\ J\subsetneq [n], k^+=\max(J), k^-=\min(J)\},
\]
and $\Delta'_Z(n;1)=\gcd(\Gamma'_Z)$, giving us
\[
\Delta_Z(n;1)=\Delta'_Z(n;1).
\]
Now, Corollary \ref{cor:Z-mu-prod} and cyclic symmetry of $\bar{\mu}$--invariants, yields \eqref{eq:D_mu=D_Z}, proving our claim.
\end{proof}
%
\begin{rem}
	We would like to  emphasize that the
	Main Theorem characterizes the value $\langle Z_{I;j}, G_L\rangle$ of the arrow polynomial $Z_{I;j}$ on an arbitrary Gauss diagram $G_L$ of a closed based link $L$, and shows that the residue class is an invariant (this property is crucial in  the forthcoming paper \cite{Komendarczyk-Michaelides:2016}).  The invariant $\bar{Z}_{n;1}(L)$ by itself can be defined in an obvious way, via Equation \eqref{eq:bar-mu-sigma} and Theorem \ref{thm:tree=mu} in terms of $Z_{n;1}(G_\ell)$ where $\widehat{\ell}$ is link homotopic to $L$.
\end{rem}



\begin{thebibliography}{10}
	
	\bibitem{Bar-Natan:1995a}
	D.~Bar-Natan.
	\newblock On the {V}assiliev knot invariants.
	\newblock {\em Topology}, 34(2):423--472, 1995.
	
	\bibitem{Bar-Natan:1995b}
	D.~Bar-Natan.
	\newblock Vassiliev homotopy string link invariants.
	\newblock {\em J. Knot Theory Ramifications}, 4(1):13--32, 1995.
	
	\bibitem{Birman-Lin:1993}
	J.~S. Birman and X.-S. Lin.
	\newblock Knot polynomials and {V}assiliev's invariants.
	\newblock {\em Invent. Math.}, 111(2):225--270, 1993.
	
	\bibitem{Chmutov-Duzhin-Mostovoy:2012}
	S.~Chmutov, S.~Duzhin, and J.~Mostovoy.
	\newblock {\em Introduction to {V}assiliev knot invariants}.
	\newblock Cambridge University Press, Cambridge, 2012.
	
	\bibitem{Chmutov-Khoury-Rossi:2009}
	S.~Chmutov, M.~C. Khoury, and A.~Rossi.
	\newblock {P}olyak-{V}iro formulas for coefficients of the {C}onway polynomial.
	\newblock {\em J. Knot Theory Ramifications}, 18(6):773--783, 2009.
	
	\bibitem{Chmutov-Polyak:2010}
	S.~Chmutov and M.~Polyak.
	\newblock Elementary combinatorics of the {HOMFLYPT} polynomial.
	\newblock {\em Int. Math. Res. Not. IMRN}, (3):480--495, 2010.
	
	\bibitem{DeTurck-Gluck-Komendarczyk-Melvin:2013a}
	D.~DeTurck, H.~Gluck, R.~Komendarczyk, P.~Melvin, C.~Shonkwiler, and D.~S.
	Vela-Vick.
	\newblock Generalized {G}auss maps and integrals for three-component links:
	toward higher helicities for magnetic fields and fluid flows.
	\newblock {\em J. Math. Phys.}, 54(1):013515, 48, 2013.
	
	\bibitem{Evans-Berger:1992}
	N.~W. Evans and M.~A. Berger.
	\newblock A hierarchy of linking integrals.
	\newblock In {\em Topological aspects of the dynamics of fluids and plasmas
		({S}anta {B}arbara, {CA}, 1991)}, volume 218 of {\em NATO Adv. Sci. Inst.
		Ser. E Appl. Sci.}, pages 237--248. Kluwer Acad. Publ., Dordrecht, 1992.
	
	\bibitem{Freedman-Krushkal:2014}
	M.~Freedman and V.~Krushkal.
	\newblock Geometric complexity of embeddings in {$\Bbb{R}^d$}.
	\newblock {\em Geom. Funct. Anal.}, 24(5):1406--1430, 2014.
	
	\bibitem{Goussarov-Polyak-Viro:2000}
	M.~Goussarov, M.~Polyak, and O.~Viro.
	\newblock Finite-type invariants of classical and virtual knots.
	\newblock {\em Topology}, 39(5):1045--1068, 2000.
	
	\bibitem{Habegger-Lin:1990}
	N.~Habegger and X.-S. Lin.
	\newblock The classification of links up to link-homotopy.
	\newblock {\em J. Amer. Math. Soc.}, 3(2):389--419, 1990.
	
	\bibitem{Habegger-Masbaum:2000}
	N.~Habegger and G.~Masbaum.
	\newblock The {K}ontsevich integral and {M}ilnor's invariants.
	\newblock {\em Topology}, 39(6):1253--1289, 2000.
	
	\bibitem{Hatcher:2002}
	A.~Hatcher.
	\newblock {\em Algebraic topology}.
	\newblock Cambridge University Press, Cambridge, 2002.
	
	\bibitem{Komendarczyk:2009}
	R.~Komendarczyk.
	\newblock The third order helicity of magnetic fields via link maps.
	\newblock {\em Comm. Math. Phys.}, 292(2):431--456, 2009.
	
	\bibitem{Komendarczyk-Michaelides:2016}
	R.~Komendarczyk and A.~Michaelides.
	\newblock Ropelength, crossing number and finite type invariants.
	\newblock {\em {\tt arXiv:1604.03870}}, 2016.
	
	\bibitem{Kotorii:2013}
	Y.~Kotorii.
	\newblock The {M}ilnor {$\overline{\mu}$} invariants and nanophrases.
	\newblock {\em J. Knot Theory Ramifications}, 22(2):1250142, 28, 2013.
	
	\bibitem{Koytcheff-Volic:2015}
	R.~Koytcheff and I.~Volic.
	\newblock Milnor invariants of string links, trivalent trees, and configuration
	space integrals.
	\newblock {\em {\tt arxiv.org:1511.02768}}, 2015.
	
	\bibitem{Kravchenko-Polyak:2011}
	O.~Kravchenko and M.~Polyak.
	\newblock Diassociative algebras and {M}ilnor's invariants for tangles.
	\newblock {\em Lett. Math. Phys.}, 95(3):297--316, 2011.
	
	\bibitem{Laurence-Stredulinsky:2000}
	P.~Laurence and E.~Stredulinsky.
	\newblock Asymptotic {M}assey products, induced currents and {B}orromean torus
	links.
	\newblock {\em J. Math. Phys.}, 41(5):3170--3191, 2000.
	
	\bibitem{Levine:1988a}
	J.~P. Levine.
	\newblock An approach to homotopy classification of links.
	\newblock {\em Trans. Amer. Math. Soc.}, 306(1):361--387, 1988.
	
	\bibitem{Levine:1988}
	J.~P. Levine.
	\newblock The {$\overline\mu$}-invariants of based links.
	\newblock In {\em Differential topology ({S}iegen, 1987)}, volume 1350 of {\em
		Lecture Notes in Math.}, pages 87--103. Springer, Berlin, 1988.
	
	\bibitem{Loday-Frabetti-Chapoton-Goichot:2001}
	J.-L. Loday, A.~Frabetti, F.~Chapoton, and F.~Goichot.
	\newblock {\em Dialgebras and related operads}, volume 1763 of {\em Lecture
		Notes in Mathematics}.
	\newblock Springer-Verlag, Berlin, 2001.
	
	\bibitem{Loday-Vallette:2012}
	J.-L. Loday and B.~Vallette.
	\newblock {\em Algebraic operads}, volume 346 of {\em Grundlehren der
		Mathematischen Wissenschaften [Fundamental Principles of Mathematical
		Sciences]}.
	\newblock Springer, Heidelberg, 2012.
	
	\bibitem{May:1972}
	J.~P. May.
	\newblock {\em The geometry of iterated loop spaces}.
	\newblock Springer-Verlag, Berlin-New York, 1972.
	\newblock Lectures Notes in Mathematics, Vol. 271.
	
	\bibitem{Mellor-Melvin:2003}
	B.~Mellor and P.~Melvin.
	\newblock A geometric interpretation of {M}ilnor's triple linking numbers.
	\newblock {\em Algebr. Geom. Topol.}, 3:557--568 (electronic), 2003.
	
	\bibitem{Michaelides:2015}
	A.~Michaelides.
	\newblock Lower bounds for ropelength via higher linking numbers and other
	finite type invariants.
	\newblock {\em --Ph.D. thesis, Tulane University}, 2015.
	
	\bibitem{Milnor:1954}
	J.~Milnor.
	\newblock Link groups.
	\newblock {\em Ann. of Math. (2)}, 59:177--195, 1954.
	
	\bibitem{Milnor:1957}
	J.~Milnor.
	\newblock Isotopy of links. {A}lgebraic geometry and topology.
	\newblock In {\em A symposium in honor of {S}. {L}efschetz}, pages 280--306.
	Princeton University Press, Princeton, N. J., 1957.
	
	\bibitem{Ostlund:2004}
	O.-P. {{\"O}}stlund.
	\newblock A diagrammatic approach to link invariants of finite degree.
	\newblock {\em Math. Scand.}, 94(2):295--319, 2004.
	
	\bibitem{Polyak:2005}
	M.~Polyak.
	\newblock Skein relations for {M}ilnor's {$\mu$}-invariants.
	\newblock {\em Algebr. Geom. Topol.}, 5:1471--1479 (electronic), 2005.
	
	\bibitem{Polyak-Viro:1994}
	M.~Polyak and O.~Viro.
	\newblock Gauss diagram formulas for {V}assiliev invariants.
	\newblock {\em Internat. Math. Res. Notices}, (11):445ff., approx.\ 8 pp.\
	(electronic), 1994.
	
	\bibitem{Polyak-Viro:2001}
	M.~Polyak and O.~Viro.
	\newblock On the {C}asson knot invariant.
	\newblock {\em J. Knot Theory Ramifications}, 10(5):711--738, 2001.
	\newblock Knots in Hellas '98, Vol. 3 (Delphi).
	
	\bibitem{Stallings:1965}
	J.~Stallings.
	\newblock Homology and central series of groups.
	\newblock {\em J. Algebra}, 2:170--181, 1965.
	
	\bibitem{Turaev:2007}
	V.~Turaev.
	\newblock Topology of words.
	\newblock {\em Proc. Lond. Math. Soc. (3)}, 95(2):360--412, 2007.
	
	\bibitem{Turaev:1976}
	V.~G. Turaev.
	\newblock The {M}ilnor invariants and {M}assey products.
	\newblock {\em Zap. Nau\v cn. Sem. Leningrad. Otdel. Mat. Inst. Steklov.
		(LOMI)}, 66:189--203, 209--210, 1976.
	\newblock Studies in topology, II.
	
	\bibitem{Vassiliev:1992}
	V.~A. Vassiliev.
	\newblock {\em Complements of discriminants of smooth maps: topology and
		applications}, volume~98 of {\em Translations of Mathematical Monographs}.
	\newblock American Mathematical Society, Providence, RI, 1992.
	\newblock Translated from the Russian by B. Goldfarb.
	
	\bibitem{Willerton:2002}
	S.~Willerton.
	\newblock On the first two {V}assiliev invariants.
	\newblock {\em Experiment. Math.}, 11(2):289--296, 2002.
	
\end{thebibliography}
\end{document}